\documentclass{amsart}

\usepackage{graphicx}

\newtheorem{theorem}{Theorem}[section]
\newtheorem*{theorem A}{Theorem A}
\newtheorem*{theorem B}{N\"olker's Theorem}
\newtheorem{lemma}{Lemma}[section]

\newtheorem{proposition}{Proposition}[section]

\theoremstyle{remark}
\newtheorem {notation}{Notation}
\theoremstyle{remark}
\newtheorem{remark}{Remark}[section]
\theoremstyle{remark}

\theoremstyle{definition}

\newtheorem{definition}{Definition}[section]
\newtheorem{example}{Example}[section]
\numberwithin{equation}{section}
\def\({\left ( }
\def\){\right )}
\def\<{\left < }
\def\>{\right >}

\begin{document}





\vspace{2cm}

\title[The Isoperimetric Inequality on Compact Rank One Symmetric Spaces]{The Isoperimetric Inequality on Compact Rank One Symmetric Spaces and Beyond}

\author{Yashar Memarian}

\address{Last institution: Department of Mathematics\\
    University of Notre-Dame}
\email{y.memarian@gmail.com}

\thanks{We thank Karsten Grove for motivating us to write this paper. The paper started when the author visited the University of Notre-Dame. We would like to thank the hospitality of the mathematics department during our stay. We also thank the referees for their extremely helpful comments which made this paper realizable.}


\date{June 2020}


\begin{abstract}
Klartag's needle decomposition technique enables one to obtain strong isoperimetric inequalities on Riemannian manifolds other than the classical known examples. As a result, in this paper, we obtain sharp isoperimetric inequalities for compact rank one symmetric spaces (CROSS). Namely, for the real projective space $\mathbb{R}P^n$, we demonstrate that the isoperimetric regions are given by either the geodesic balls or tubes around some $\mathbb{R}P^k\subset\mathbb{R}P^n$. For the complex projective space $\mathbb{C}P^n$, the isoperimetric regions are given by either the geodesic balls or tubes around some $\mathbb{C}P^k\subset\mathbb{C}P^n$. And for the quaternionic projective space, the isoperimetric regions are given by either the geodesic balls or tubes around some $\mathbb{H}P^k\subset\mathbb{H}P^n$.
\end{abstract}
\maketitle

\section{Introduction}

Isoperimetric problems are some of the oldest problems in geometry. Given a space, one looks for \emph{domains} of a given volume with the least boundary \emph{area}. It is known that in model spaces (\emph{i.e.} Euclidean spaces, spheres and hyperbolic spaces), for every given number $v$, the \emph{intrinsic balls} with volume $v$ have the least surface area among every domain with the same volume. When we leave the world of model spaces, or when we are dealing with geometric spaces with boundary, the solution to isoperimetric problems has provided real difficulties. Of course, when the number $v$ is small enough, for spaces which locally \emph{look} like Euclidean spaces (for instance manifolds), one expects that the solution of isoperimetric problems still would be the metric balls. However, for larger $v$ and shapes with non-constant curvature or with non-smooth boundaries, the isoperimetric problem in general is very hard to solve. There are several books and surveys related to isoperimetric problems, for instance \cite{ros}, \cite{oserman}, \cite{groiso}, \cite{chavel}, \cite{figali}.

In \cite{gromil}, Gromov-Milman prove a general isoperimetric problem on spheres with non-necessarily canonical Riemannian volume. They use their result to obtain an isoperimetric inequality for unit spheres of uniformly-convex Banach spaces (for example $L^p$-unit spheres). Their method relies on the geometry and topology of the sphere. They use a powerful technique (known as the localization method) in order to prove their result(s). Obtaining similar results for more general manifolds using the localization method was not possible since one did not have such tools for spaces other than model spaces (even on the canonical hyperbolic space we are not aware of any result for which the localization method has been used). 

Recently, Klartag in \cite{kl} proves a localization theorem on every closed Riemannian manifold. The topic of this paper concerns utilizing Klartag's  localization results in order to prove sharp isoperimetric inequalities for compact rank one symmetric spaces, namely the real, complex and quaternionic projective spaces.

We begin by recalling Klartag's results on needle decomposition and its link to the isoperimetric problems. In further sections we concentrate on studying the isoperimetric problem on real, complex and quaternionic projective spaces.

\section{Klartag's Generalisation of Gromov-Milman Isoperimetric Inequality on Riemannian Manifolds}

In this manuscript, we are mainly interested in closed, smooth $n$-dimensional Riemannian manifolds with some lower bounds on the Ricci curvature. A Riemannian manifold $M$ satisfies the curvature-dimension condition $CD(k,N)$, if the dimension of $M$ is at most equal to $N$ and if for every point $m\in M$ and every tangent vector $v\in T_{m}M$ we have:
\begin{equation} \label{eqn:ric}
Ricci(v,v)\geq k.g(v,v),
\end{equation}
where $g$ is the metric tensor of $M$ and $Ricci$ is the Ricci tensor of $M$. When for a Riemannian manifold $M$, there exists a $k\in\mathbb{R}$ such that for every point $m\in M$ and every tangent vector $v\in T_{m}M$, inequality (\ref{eqn:ric}) holds, we simply write that on $M$ we have:
\begin{eqnarray*}
Ricci\geq k.
\end{eqnarray*}
The curvature-dimension condition and its extension to the case of weighted Riemannian manifolds were first defined in the pioneer work \cite{be}. 

The volume Riemannian volume of $M$ is denoted by $vol_n$ emphasizing the dimension of $M$.

We recall the (metric-measure) invariant \emph{separation distance} introduced by Gromov in \cite{grobook} (and motivated by ideas in \cite{gromil}).

\begin{definition}[Separation Distance] \label{duf}
Let $\kappa_1, \kappa_2>0$ be such that
\begin{eqnarray*}
\kappa_1+\kappa_2<1.
\end{eqnarray*}
The \emph{separation distance} on $M$ (with respect to $\kappa_1$ and $\kappa_2$), denoted by:
\begin{eqnarray*}
Sep(M,\kappa_1,\kappa_2),
\end{eqnarray*}
is the supremum of those $\delta$, where  for $i=1,2$, subsets $A_i\subset M$ exist with $\mu(A_i)\geq \kappa_i$ and $dist(A_1,A_2)\geq \delta$. Here 
\begin{eqnarray*}
dist(A_1,A_2)=\inf_{(a_1,a_2)\in A_1\times A_2} d(a_1,a_2),
\end{eqnarray*}
where $d(.)$ is the distance induced by the Riemannian metric on $M$. We say that the sets $A_i$ (for $i=1,2$) with $\mu(A_i)=\kappa_i$, realize the separation distance if $dist(A_1,A_2)=Sep(M,\kappa_1,\kappa_2)$.
\end{definition}
\begin{remark}
\begin{itemize}
\item We denote the distance between two sets $A_1,A_2\subset M$ defined in definition \ref{duf} by:
\begin{eqnarray*}
d(A_1,A_2)=dist(A_1,A_2).
\end{eqnarray*}
This distance should not be confused with the Hausdorff distance.
\item One can define the separation distance with respect to several (positive real) numbers $\kappa_1,\cdots,\kappa_k$, in the same manner as defined with respect to two numbers. For the purpose of this paper, we only require this definition with respect to two numbers.
\end{itemize}
\end{remark}

We now need the definition of a needle on $n$-dimensional Riemannian manifolds $M$ with $Ricci(M)\geq Ricci(\mathbb{S}^n)$. Since we are using Klartag's work on needle decomposition of Riemannian manifolds, we present the definition of the needle exactly as it is defined in \cite{kl}:

\begin{definition}[Needles] \label{defneedle}
Let $M$ be an $n$-dimensional Riemannian manifold. A $CD(k,n)$-needle on $M$ is a measure $\nu$ on $M$ such that the following holds:
\begin{itemize}
\item There exists a non-empty open set $A\subset\mathbb{R}$ and a smooth function $\psi:A\to\mathbb{R}$.
\item The function $\psi^{1/n-1}$ satisfies the following differential inequality on $A$:
\begin{equation} \label{eqn:sinconc}
(\psi^{\frac{1}{n-1}})''+(\frac{k}{n-1})\psi^{\frac{1}{n-1}}\leq 0.
\end{equation}
\item There exists a minimizing geodesic $\gamma:A\to M$ such that:
\begin{eqnarray*}
\nu=\gamma_{*}(\psi(x)dx).
\end{eqnarray*}
Hence the measure $\nu$ defined on $M$ is the push-forward of the measure $\psi(x)dx$ on $A$ which has a density function $\psi$ with respect to the Lebesgue measure $dx$.
\end{itemize}
A smooth function satisfying inequality (resp. equality) (\ref{eqn:sinconc}) for $k=n-1$ is called a $\sin^{n-1}$-concave (resp. affine) function. A measure $\nu$ which has a $\sin^{n-1}$ density function is called a $\sin^{n-1}$-concave measure.
\end{definition}

\begin{remark}
\begin{itemize}
\item Sometimes it may be important to recall the parameter interval $A$ on which the pull-back of a needle is defined and supported. In such case, we may denote a $CD(k,n)$-needle by $(I,\nu)$ where $I\subset\mathbb{R}$ is a non-empty open connected subset of $\mathbb{R}$ and the pull-back of the measure $\nu$ is supported on $I$. This is useful when one considers needles on intervals of $\mathbb{R}$ \emph{i.e.} geodesic segments of $\mathbb{R}$.
\item In definition \ref{defneedle}, when $k=n-1$, the needles may be referred to as $\sin^{n-1}$-concave needles (or measures). 
\item The inequality (\ref{eqn:sinconc}) is a \emph{local} definition of the $\sin^n$-concave functions. One (equivalently) may define a $\sin^n$-concave function on an interval $A\subset\mathbb{R}$ as a function satisfying the following inequality:
\begin{equation} \label{eqn:ohtz2}
f(\frac{x_1+x_2}{2})^{1/n}\geq \frac{f(x_1)^{1/n}+f(x_2)^{1/n}}{2\cos(\frac{|x_2-x_1|}{2})},
\end{equation}
for every $x_1,x_2\in A$. See \cite{vil} for a proof of the equivalence between inequalities (\ref{eqn:ohtz2}) and (\ref{eqn:sinconc}).
\item According to definition \ref{defneedle}, we may view needles as intrinsic metric-measure spaces (\emph{i.e.} not necessarily embedded on a manifold $M$) where the geometric space is an interval of $\mathbb{R}$, the metric being the Euclidean metric and the measure being the $\sin^n$-concave probability measure defined on this interval. On the other hand, it is important to bear in mind that needles inherit the geometry of the manifold on which they are \emph{embedded}.
\end{itemize}
\end{remark}

The following definition is handy since we shall be dealing with probability measures all along this manuscript:
\begin{definition}[Normalizing Constant]
Let $f(x)dx$ be a measure supported on an interval $A\subset\mathbb{R}$ where $dx$ is the Lebesgue measure. We say $C\in\mathbb{R}$ is the normalizing constant for $f(x)dx$, if $Cf(x)dx$ is a probability measure supported on $A$.
\end{definition}

We present some classical examples to get familiarized with some needles:

\begin{example}
On the canonical sphere $\mathbb{S}^n$, let $\{x,-x\}$ be two diametrically opposite points. Let $\gamma$ be a maximal geodesic from $x$ to $-x$. This geodesic being parametrized by the interval $(0,\pi)$ which is enhanced with the normalized measure $C\cos(t)^{n-1}dt$. Then the push-forward of this measure on $\mathbb{S}^n$ is an example of a needle on the canonical sphere. 
\end{example}
\begin{example}
A $\sin^n$-affine probability measure $\nu$ defined on an open connected interval $A\subset\mathbb{R}$ is defined by $\nu=(C_1\sin(t)+C_2\cos(t))^ndt$ for some $C_1,C_2\in\mathbb{R}$ where $dt$ is the Lebesgue measure.
\end{example}

We have mentioned that one could consider needles defined intrinsically, and viewing them as such a needle is a metric-measure space defined on its own. But how are the needles constructed on a manifold $M$ (with $Ricci(M)\geq Ricci(\mathbb{S}^n$)?

For the case $M=\mathbb{S}^n$, the construction of a needle is explained in \cite{gromil}. Indeed, the authors show that needles are obtained as a \emph{limit} of $S_1\supset S_2\supset\cdots$, where each $S_i$ is a (geodesically) convex subset of $\mathbb{S}^n$. Here, the limit is a weak limit of the normalized volume of the $S_i$ (for $i=1\cdots \infty$).

When $M$ is no longer the canonical sphere, the definition and construction of needles are explained in \cite{kl}. The strategy of the author in \cite{kl} is to define some geometric objects called \emph{needle candidates}, and then show that every needle candidate is a needle in the sense of definition \ref{defneedle}. We present the definition of the needle candidates below and shall understand needles as some \emph{measures} which are tied to the geometry of $M$.

Let $\gamma:[0,a]\to M$ be a minimizing geodesic on $M$ which is parametrized by its arclength. One may use the Riemannian connection on $M$ and define the covariant derivative of any smooth vector field $J$ along $\gamma$. A smooth vector field $J$ along $\gamma$ is a smooth vector field $J$ such that $J(t)\in T_{\gamma(t)}M$ (the tangent space of $M$ at point $\gamma(t))$. The covariant derivative of $J$ along $\gamma$ is denoted by
\begin{eqnarray*}
J'=\nabla_{\gamma'}J,
\end{eqnarray*}
where $\gamma'$ is the tangent vector field along $\gamma$.

A vector field along $\gamma$ is called a \emph{Jacobi field} if the following equation is satisfied for every $t\in [0,a]$:
\begin{eqnarray*}
J''(t)=R(\gamma'(t),J(t))\gamma'(t),
\end{eqnarray*}
where $R$ is the Riemann curvature tensor and for $X,Y,Z\in TM$ is defined by:
\begin{eqnarray*}
R(X,Y)Z=\nabla_{X}\nabla_{Y}Z-\nabla_{Y}\nabla_{X}Z-\nabla_{[X,Y]}Z.
\end{eqnarray*}
For every smooth function $f\in C^{\infty}(M)$, $[X,Y]$ is defined by
\begin{eqnarray*}
[X,Y](f)=X(Y(f))-Y(X(f)).
\end{eqnarray*}

With this background, we are now ready to give the definition of needle candidates presented in \cite{kl}:

\begin{definition}[Needle Candidates] \label{nc}
A measure $\nu$ on a Riemannian manifold $M$ is called a \emph{needle candidate} on $M$ if there exists a non-empty subset $(a,b)\subset\mathbb{R}$ with $ a,b\in\mathbb{R}\cup\{-\infty,+\infty\}$, a measure $\mu$ on $(a,b)$, a minimizing geodesic $\gamma:(a,b)\to M$ and Jacobi fields $J_1(t),\cdots,J_{n-1}(t)$ along $\gamma$ with the following properties:
\begin{itemize}
\item The measure $\nu$ is the push-forward of the measure $\mu$ under the map $\gamma$.
\item Denote $J_n=\gamma'$. Then the measure $\mu$ is absolutely continuous with respect to the Lebesgue measure in $(a,b)$ and its density is proportional to 
\begin{eqnarray*}
t\to\sqrt{\det(<J_i(t),J_k(t)>)_{i,k=1,\cdots,n}}.
\end{eqnarray*}
\item There exists $t\in(a,b)$ with
\begin{eqnarray*}
<J_i(t),\gamma'(t)>\:=\:<J'_i(t),\gamma'(t)>=0\:(i=1\cdots,n-1),
\end{eqnarray*}
and
\begin{eqnarray*}
<J'_i(t),J_k(t)>\:=\:<J'_k(t),J_i(t)>\: (i,k=1,\cdots,n-1).
\end{eqnarray*}
\item Either for all $t\in(a,b)$ the vectors
\begin{eqnarray*}
J_1(t),\cdots J_{n-1}(t)\in T_{\gamma(t)}M
\end{eqnarray*}
are linearly independent or else for all $t\in(a,b)$ these vectors are linearly dependent.
\end{itemize}
\end{definition}
\begin{remark}
In \cite{kl} where a  needle candidate is defined, there exists a Lipschitz function $u$ such that the needle candidate is defined \emph{along} the integral curves of the gradient of $u$. We refer the reader to \cite{kl} for all details.
\end{remark}

The following, proved in \cite{kl}, relates the needle candidates to the needles of definition \ref{defneedle}: 
\begin{proposition} \label{pff}
Let $M$ be a Riemannian manifold with $Ricci(M)\geq Ricci(\mathbb{S}^n$). Let $\nu$ be a needle candidate defined in \ref{nc}. Then $\nu$ is a needle as defined in \ref{defneedle}.
\end{proposition}

We would like to extend the definition of \emph{separation distance} to the class of needles and/or needle candidates.

Should one follow the intrinsic point of view, one sets $n,k\geq 0$ and considers the class of $CD(k,n)$-needles. We consider probability measures and then this is a cone in the space of probability measures with support in $\mathbb{R}$ (see \cite{growst}). Let $\mathcal{A}$ be a class of $CD(k,n)$-(probability) needle. Then one defines the \emph{needle separation distance} with respect to $\mathcal{A}$ as follows:
\begin{definition}
Given two positive real numbers $\kappa_1,\kappa_2>0$, such that
\begin{eqnarray*}
\kappa_1+\kappa_2<1,
\end{eqnarray*}
the \emph{needle separation distance} with respect to $\mathcal{A}$ is defined by
\begin{eqnarray*}
N_{\mathcal{A}}(\kappa_1,\kappa_2)=\sup_{\nu\in\mathcal{A}}Sep(\nu,\kappa_1,\kappa_2),
\end{eqnarray*}

where $Sep(\nu,\kappa_1,\kappa_2)$ is the separation distance on the metric-measure space $(I,\nu)$, and the supremum is taken over every needle in $\mathcal{A}$.

We say $\nu$ realizes the needle separation distance $N_{\mathcal{A}}(\kappa_1,\kappa_2)$ if
\begin{eqnarray*}
Sep(\nu,\kappa_1,\kappa_2)=N_{\mathcal{A}}(\kappa_1,\kappa_2).
\end{eqnarray*}
Moreover the sets $I_i\subseteq I$ with $\nu(I_i)=\kappa_i$ (for $i=1,2$) realize the needle separation distance if 
\begin{eqnarray*}
d(I_1,I_2)=N_{\mathcal{A}}(\kappa_1,\kappa_2).
\end{eqnarray*}
\end{definition}

We are interested in studying the isoperimetric problem on a Riemannian manifold. For this purpose we need the separation distance for the needle candidates, as the latter reflects in some ways the isoperimetric properties of $M$.

\begin{definition}[Needle Candidate Separation Distance]
Let $M$ be a Riemannian manifold of dimension $n$ with $Ricci(M)\geq Ricci(\mathbb{S}^n)$. Given two positive real numbers $\kappa_1,\kappa_2>0$ such that
\begin{eqnarray*}
\kappa_1+\kappa_2<1,
\end{eqnarray*}
the \emph{needle candidate separation distance} is defined as follows:
\begin{eqnarray*}
N_{M}(\kappa_1,\kappa_2)=\sup_{\nu}Sep(\nu,\kappa_1,\kappa_2).
\end{eqnarray*}
Here, $Sep(\nu,\kappa_1,\kappa_2)$ is the separation distance on the metric-measure space $(\gamma,\nu)$ (\emph{i.e.} a geodesic segment and a probability measure defined upon it). The supremum is taken over all needle candidates $\nu$ which can be defined on $M$. We say $I_i\subseteq I$ with $\nu(I_i)=\kappa_i$ (for $i=1,2$) realize the needle candidate separation distance if 
\begin{eqnarray*}
d(I_1,I_2)=N_{M}(\kappa_1,\kappa_2).
\end{eqnarray*}
Furthermore we say $A_i\subseteq M$ with $\mu(A_i)=\kappa_i$ realize the \emph{needle candidate separation distance} if 
\begin{eqnarray*}
d(A_1,A_2)=N_{M}(\kappa_1,\kappa_2). 
\end{eqnarray*}
\end{definition}


In the figure below, we review the definition of the separation and needle separation distances on an example:

\begin{center}
\includegraphics[width=2.5in]{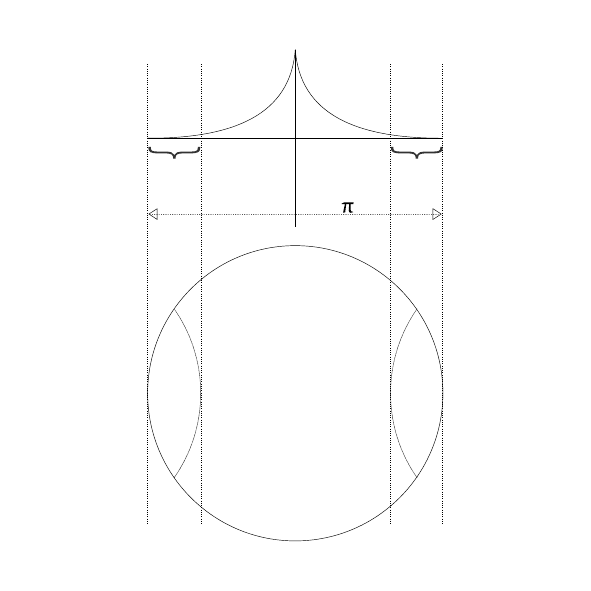}
\end{center}

In the upper part, one can find the graph of the function $C\cos(t)^n$ on the interval $(-\pi/2,+\pi/2)$, for some $n>1$ where $C$ is the normalizing constant. On the left and right hand sides of the interval $(-\pi/2,+\pi/2)$, we find two sub-intervals with two given measures (say $\kappa_1$ and $\kappa_2$). The separation distance for the needle $((-\pi/2,+\pi/2), C\cos(t)^n dt)$ is the distance between these two intervals where the distance is defined in definition \ref{duf} (and should not be confused with the Hausdorff distance). Below the graph of this needle, one can observe a disc which is seen as the projection of a (hemi)-sphere. The diameter of the sphere is equal to the diameter of the needle pictured above it. The left and right hand sides of this disc are the projection of two spherical balls. Each spherical ball has a measure equal to the measure of the interval on the needle shown above it. The separation distance for this sphere is now defined to be the distance (in the sense of definition \ref{duf}) between these two spherical balls.

We are now ready to present the following strong Theorem which is due to Klartag:
\begin{theorem} \label{main}
Let $M$ be a closed, smooth Riemannian manifold where $Ricci(M)\geq Ricci(\mathbb{S}^n)$. Given $\kappa_1,\kappa_2>0$ such that
\begin{eqnarray*}
\kappa_1+\kappa_2<1,
\end{eqnarray*}
we have:
\begin{eqnarray*}
Sep(M,\kappa_1,\kappa_2)\leq N_{M}(\kappa_1,\kappa_2).
\end{eqnarray*}
\end{theorem}

Similar to the proof in \cite{gromil}, this Theorem is proved by using the localization technique.

Gromov-Milman prove this theorem on the sphere $\mathbb{S}^n$, in which the topology of the sphere enables one to construct such needles. However, on general Riemannian manifolds, one can not provide needle decomposition in the same way as one provides needle decomposition on $\mathbb{S}^n$. To remedy this issue, Klartag (in \cite{kl}) presents a localization theorem on every closed Riemannian manifold. Theorem \ref{main} has very important consequences, when applied on specific examples.

We provide a proof of Klartag's separation distance Theorem \ref{main} here and further present a corollary of this Theorem which links this result to the study of (sharp) isoperimetric inequalities on Riemannian manifolds. 

\begin{theorem}[Klartag] \label{klartag}
Let $M$ be a closed, smooth Riemannian manifold with $Ricci(M)\geq Ricci(\mathbb{S}^n)$. Let $\mu$ be the Riemannian volume on $M$. Let $f$ be a $\mu$-integrable function such that 
\begin{eqnarray*}
\displaystyle\int_{M}f(x)d\mu(x)=0.
\end{eqnarray*}
Then, there exist a partition $\Pi$ of $M$, a measure $\nu$ on $\Pi$, a family $\{\mu_{\pi}\}_{\pi\in\Pi}$ of measures on $M$ such that:
\begin{itemize}
\item For any Lebesgue-measurable set $U\subseteq M$ we have:
\begin{eqnarray*}
\mu(A)=\displaystyle\int_{\Pi}\mu_{\pi}d\nu(\pi).
\end{eqnarray*}
\item For $\nu$-almost every $\pi\in\Pi$, the set $\pi\subset M$ is the image of a minimizing geodesic. The measure $\mu_{\pi}$ is supported on $\pi$, and either $\pi$ is a singleton or a needle candidate in the sense of definition \ref{nc}.
\item For $\nu$-almost any $\pi\in\Pi$ we have:
\begin{eqnarray*}
\displaystyle\int_{\pi}f d\mu_{\pi}=0.
\end{eqnarray*}
\end{itemize}
\end{theorem}
\begin{remark}
Consult \cite{kl} for a \emph{precise} understanding on how the \emph{partitions} are defined and constructed. In \cite{kl} Klartag proves that the partition of Theorem \ref{klartag} is obtained as the integral curves of the gradient of some function $u:M\to\mathbb{R}$.
\end{remark}

\subsection{Proof of Theorem \ref{main}}

\begin{proof}

Let $ \kappa_1,\kappa_2>0$  such that
\begin{eqnarray*}
\kappa_1+\kappa_2<1
\end{eqnarray*}
be given. Since $M$ is compact, there exist $U_i$ which realize the separation distance (with respect to $\kappa_i$) for $i=1,2$. According to Theorem \ref{klartag}, there exists a partition of $M$ into needle candidates such that for almost every needle candidate $\nu$ we have:
\begin{eqnarray*}
\nu(U_i)=\kappa_i.
\end{eqnarray*}
Indeed, assume $\mu(U_1)=\kappa\mu(U_2)$. Hence the integral of the function $\xi_{U_1}-\kappa\xi_{U_2}$ is equal to zero on $M$. Here $\xi$ is the indicatrice function. We apply Theorem \ref{klartag} with respect to this function.

Since we have a partition of $M$ into needle candidates (which satisfy the measure assumption above), there exists (at least) a needle $\nu$ in this partition such that 
\begin{eqnarray*}
d((I\cap U_1),(I\cap U_2))\geq d(U_1,U_2),
\end{eqnarray*}
where $\gamma:A\to M$ and $\gamma(A)=I\subset M$.

Hence by definition we have:
\begin{eqnarray*}
N_{M}(\kappa_1,\kappa_2)&\geq& d((I\cap U_1),(I\cap U_2))\\
                    &\geq& d(U_1,U_2)\\
                    &=& Sep(M,\kappa_1,\kappa_2).
\end{eqnarray*}

This completes the proof of Theorem \ref{main}.

\end{proof}
\begin{remark}
Theorem \ref{main} is valid for the class of weighted Riemannian manifolds satisfying the generalized curvature-dimension property. Consult \cite{kl} for this purpose.
\end{remark}
\begin{notation}
For $A\subset M$ and for $\varepsilon>0$, $A+\varepsilon$ stands for the $\varepsilon$-neighborhood of the set $A$.
\end{notation}
Let us see how Theorem \ref{main} may be used for solving isoperimetric (type) problems. This is explained in \cite{kl} and \cite{kl2}.
\begin{proposition} \label{impoo}
Let $M$ be a smooth closed $n$-dimensional Riemannian manifold with $Ricci(M)\geq Ricci(\mathbb{S}^n)$. Let $\mu$ be the normalized Riemannian volume and $\varepsilon>0$ be fixed. For every $0<v\leq 1$, and for every open set $A$ with measure equal to $v$, let 
\begin{eqnarray*}
w(A)=1-\mu(A+\varepsilon).
\end{eqnarray*}
Suppose $(B,B_1)\subseteq M\times M$ exists which realizes the \emph{needle separation distance} $N(v,w(A))$. Then we have: 
\begin{eqnarray*}
\mu(A+\varepsilon)\geq \mu(B+\varepsilon).
\end{eqnarray*}
\end{proposition}

\begin{proof}
Assume $A_1$ is an open set on $M$ with a measure equal to $v$. Let $\varepsilon>0$ be given and let $A_2=M\setminus (A_1+\varepsilon)$ and $A_3=(A_1+\varepsilon)\setminus A_1$. Let $\mu(A_2)=w$.

According to Theorem \ref{main} we have:
\begin{eqnarray*}
Sep(M,v,w)\leq N_{M}(v,w).
\end{eqnarray*}
By assumption, let $(B_1,B_2)\subseteq M\times M$ be such that $\mu(B_1)=v$, $\mu(B_2)=w$ and
\begin{eqnarray*}
d(B_1,B_2)=N_{M}(v,w).
\end{eqnarray*}

Therefore we obtain:
\begin{eqnarray*}
d(A_1,A_2)&=&\varepsilon \\
             &\leq& d(B_1,B_2).
\end{eqnarray*}
This means $\mu(M\setminus(B_1+\varepsilon))\geq \mu(A_2)$. And thus
\begin{eqnarray*}
\mu(A_1+\varepsilon)\geq \mu(B_1+\varepsilon).
\end{eqnarray*}
 The proof therefore follows.
\end{proof}


\section{Isoperimetric properties of Needles}

\bigskip Considering the $CD(k,n)$-needles as metric-measure spaces of their own, there is the natural problem of characterizing the needle(s) which attains $N(\kappa_1,\kappa_2)$ for a given $\kappa_1,\kappa_2\in(0,1)$. This question only makes sense if we define $\mathcal{A}$ (the domain in the cone of $CD(k,n)$-needles) on which we take the supremum of $Sep(\nu,\kappa_1,\kappa_2)$. Here we will be dealing with the \emph{most} general isoperimetric properties of needles. In other words, given $n\in\mathbb{N}$, we will be studying $N(\kappa_1,\kappa_2)$ when $\mathcal{A}$ is the entire cone of $CD(k,n)$-needles. Most of the techniques which follow shall also be used to study $N_{M}(\kappa_1,\kappa_2)$ for some specific (symmetric) Riemannian manifolds. 

Studying the properties of a needle is equivalent to studying the properties of their densities of their pull-back on an interval of $\mathbb{R}$. For simplicity, we take $k=n-1$ here and present some properties of the $\sin^n$-concave functions (measures).

There exists a classical way to study the properties of a $\sin^n$-concave function which is related to the local definition of such functions, namely the differential inequality (\ref{eqn:sinconc}). This is provided by the Sturm-Liouville theory (about differential equations): Let $u$ satisfy inequality (\ref{eqn:sinconc}) (resp. inequality) on the interval $[a,b]$. Let $v$ satisfy the \emph{equality} in (\ref{eqn:sinconc}) (resp. equality) on the same interval $[a,b]$, where $u(a)=v(a)$ and $u(b)=v(b)$ (or $u(a)=v(a)$ and $u'(a)=v'(a)$). Then $u\leq v$ on $[a,b]$.

In order to obtain integral formulae using the above method, one may use the following Lemma (assigned to Gromov and used to prove  the Bishop-Gromov inequality):
\begin{lemma} \label{gr}
Let $f, g$ be real positive function on $[0,A]$ where $A\in\mathbb{R}_{+}$ such that
\begin{eqnarray*}
\frac{f}{g}
\end{eqnarray*}
is a non-increasing function, then
\begin{eqnarray*}
\frac{\displaystyle\int_{0}^{x} f(t)dt}{\displaystyle\int_{0}^{x}g(t) dt}
\end{eqnarray*}
is a non-increasing function on $[0,A]$.
\end{lemma}

There exists a more geometric way (related to the \emph{global} properties of $\sin^n$-concave functions, namely inequality (\ref{eqn:ohtz2})) to study the properties of $\sin^n$-concave functions. This is worth explaining here:
\begin{itemize}
\item Let $\mu$ be a measure which has a density function $f$ and suppose $f$ is a $\sin^n$-concave function. By definition the function $g=f^{1/n}$ is a $\sin$-concave function. 
\item Consider the \emph{support} of $\mu$ and suppose the support is $I\subseteq [0,\pi]$. 
\item \emph{Transport} $I$ on the (half)-unit circle (denoted by $\mathbb{S}^1_{+}$), which is the unit circle in $\mathbb{R}^2$ with non-negative $y$-coordinates. By transport we mean a unit speed parametrization map $I\to\mathbb{S}^1_{+}$.
\item Consider the \emph{cone} over the arc $I$ in $\mathbb{R}^2_{+}$  denoted by $co(I)$, defined by:
\begin{eqnarray*}
co(I)= \{\cup(tI)\:\vert\, 0\leq t\leq 1\}.
\end{eqnarray*}
\item The definition of a $\sin$-concave function can now be given as follows: there exists a function $G$ in $C(I)$ such that $G$ is a concave function and $G$ is $1$-homogeneous (\emph{i.e.} $G(\lambda x)=\lambda x$ for $\lambda\geq 0$) and $G(x)=g(x)$ for $x\in I$.
\item Therefore, to understand the properties of a $\sin$-concave funcion/measure, one could study the properties of its extension on a cone in $\mathbb{R}^2$ where the extended function is a $1$-homogeneous concave function.
\end{itemize}

The above approach is followed in \cite{memwst}, and most of the properties of $\sin^n$-concave measures are obtained in this way. 

We present the main properties in the following Lemma and omit the proof. For complete proof one could consult \cite{memwst}.

\begin{lemma} \label{comp}
\begin{itemize}
\item
Let $0<\varepsilon\leq \pi/2$. Let $\tau\geq \varepsilon$. Let $f$ be a non-negative $\sin^n$-concave function on $[0,\tau]$, which attains its maximum at $0$. Let $h(t)=C\cos(t)^n$ where $C$ is chosen such that $f(\varepsilon)=h(\varepsilon)$. 
\begin{itemize}
\item Then we have
\begin{equation*}
\begin{cases}
  f(x)\geq h(x) & \text{for $x\in[0,\varepsilon]$}, \\
 f(x)\leq h(x) & \text{for $x\in[\varepsilon,\tau]$}.
\end{cases}
\end{equation*}
And $\tau\leq \pi/2$.

\item For every $k\geq 0$ and $\varepsilon\leq \pi/2$, we have :
\begin{equation} \label{eqn:jun}
\frac{\displaystyle\int_{0}^{\min\{\varepsilon,\tau\}}f(t)\sin(t)^k dt}{\displaystyle\int_{0}^{\tau}f(t)\sin(t)^k dt}\geq \frac{\displaystyle\int_{0}^{\varepsilon}\cos(t)^n\sin(t)^k dt}{\displaystyle\int_{0}^{\pi/2}\cos(t)^n\sin(t)^k dt}.
\end{equation}
\end{itemize}
\item Let $(s,k)\in\mathbb{N}\times\mathbb{N}$ be such that $s\leq k$. Then every $\sin^k$-concave function is also a $\sin^s$-concave function. Moreover, if $f$ is $\sin^m$-concave, $g$ is $\sin^n$-concave, then $fg$ is $\sin^{m+n}$-concave.
\item The above properties hold for functions which satisfy the differential inequality of (\ref{eqn:sinconc}). Indeed, for such functions, everything above is to compare with \emph{model} functions: 
\begin{eqnarray*}
\frac{C\sin^n(\sqrt{\lambda}t)}{\sqrt{\lambda}},
\end{eqnarray*}
which are the functions satisfying the differential equality in (\ref{eqn:sinconc}).
\item Let $I\subset\mathbb{R}$ be a connected closed interval. Let $\mu=f(t)dt$ be a probability measure where $f$ is a $\sin^n$-concave function. Let $\kappa_1,\kappa_2>0$ be such that 
\begin{eqnarray*}
\kappa_1+\kappa_2<1.
\end{eqnarray*}
Then, the separation distance 
\begin{eqnarray*}
Sep(\mu,\kappa_1,\kappa_2),
\end{eqnarray*}
is realized by $A_1,A_2\subset I$ with
\begin{eqnarray*}
\mu(A_1)&=&\kappa_1 \\
\mu(A_2)&=&\kappa_2.
\end{eqnarray*}
Moreover, both $A_1$ and $A_2$ are \emph{connected} intervals.
\item Let $\mu_1=C_1 f(t)dt$ be a probability measure on $(0,L_1)$ where $f$ is a $\sin^n$-concave function. Let $\mu_2=C_2 f(t)dt$ be a probability measure on $(0,L_2)$ where
\begin{eqnarray*}
L_2\geq L_1.
\end{eqnarray*}
Then for every $\kappa_1,\kappa_2>0$ such that
\begin{eqnarray*}
\kappa_1+\kappa_2<1,
\end{eqnarray*}
we have:
\begin{equation} \label{eqn:big}
Sep(\mu_2,\kappa_1,\kappa_2)\geq Sep(\mu_1,\kappa_1,\kappa_2).
\end{equation}
\end{itemize}
\end{lemma}

\begin{proof}
\begin{itemize}
\item For the proof of parts $1$ to $3$, one may consult \cite{memwst}, \cite{growst} and \cite{vil}.
\item The existence and connectivity of the intervals which realize the separation distance is obtained from a classical compacity argument followed by the concavity properties of the $\sin^n$-concave functions.
\item For the last part, let 
\begin{eqnarray*}
0<\varepsilon\leq \delta.
\end{eqnarray*}
We use Lemma \ref{gr} to obtain:
\begin{eqnarray*}
\frac{\displaystyle\int_{0}^{\varepsilon}f(t)dt}{\displaystyle\int_{0}^{L_1}f(t)dt}\geq \frac{\displaystyle\int_{0}^{\delta}f(t)dt}{\displaystyle\int_{0}^{L_2}f(t)dt}.
\end{eqnarray*}
By using the above inequality the proof of (\ref{eqn:big}) follows.
\end{itemize}
\end{proof}

The general isoperimetric property of $CD(n-1,n)$-needles is presented in the following:

\begin{proposition} \label{supra}
Let $\kappa_1,\kappa_2>0$ be such that
\begin{eqnarray*}
\kappa_1+\kappa_2<1.
\end{eqnarray*}
Let $\mu=f(t)dt$ be a probability measure where $f$ is a $\sin^n$-concave function and $\mu$ is supported on the interval $I\subset\mathbb{R}$. Then
\begin{eqnarray*}
Sep(\mu,\kappa_1,\kappa_2)\leq Sep(C\sin(t)^n dt,\kappa_1,\kappa_2),
\end{eqnarray*}
where $C$ is the normalizing constant and the probability measure $C\sin(t)^ndt$ is supported on the interval $(0,\pi)$.
\end{proposition}
\begin{proof}
From Lemma \ref{comp}, we know that the length of the support of the measure $\mu$ is at most equal to $\pi$. Therefore, we \emph{translate} (\emph{i.e} change of variable $t\to t+a$) the function $f$ such that the point on which the (\emph{unique}) maximum is attained, coincides with the point $\pi/2$ on which the maximum of $\sin(t)^n$ is attained. Let $\varepsilon>0$ be such that:
\begin{eqnarray*}
\frac{\displaystyle\int_{t_1}^{t_2}\sin(t)^ndt}{\displaystyle\int_{0}^{\pi} \sin(t)^n dt}=1-(\kappa_1+\kappa_2),
\end{eqnarray*}
where $t_1=\pi/2-\varepsilon$ and $t_2=\pi/2+\varepsilon$.

From Lemma \ref{comp} namely equation (\ref{eqn:jun}) (when one takes $k=0$ in this formula), we have:
\begin{eqnarray*}
1-(\kappa_1+\kappa_2) &=&\frac{\displaystyle\int_{t_1}^{t_2}\sin(t)^ndt}{\displaystyle\int_{0}^{\pi}\sin(t)^n dt} \\
                    &\leq& \frac{\displaystyle\int_{t_1}^{t_2}f(t+a)dt}{\displaystyle\int_{0}^{\pi} f(t+a)dt}.
\end{eqnarray*}
Hence, there exists a $\varepsilon_1\leq \varepsilon$ for which we have
\begin{eqnarray*}
\frac{\displaystyle\int_{s_1}^{s_2}f(t+a)dt}{\displaystyle\int_{0}^{\pi} f(t+a)dt}= 1-(\kappa_1+\kappa_2),
\end{eqnarray*}
where $s_1=\pi/2-\varepsilon_1$ and $s_2=\pi/2+\varepsilon_1$.

This proves that:
\begin{eqnarray*}
Sep(\mu,\kappa_1,\kappa_2)\leq Sep(C\sin(t)^ndt,\kappa_1,\kappa_2).
\end{eqnarray*}
This  ends the proof of the proposition.
\end{proof}

\subsection{Recollection on Spherical Isoperimetric Problem and the Levy-Gromov Isoperimetric inequality}

With the tools presented in the previous sections, namely Proposition \ref{impoo} and Proposition \ref{supra}, we can already recover the sharp isoperimetric inequality on the canonical sphere as well as the Levy-Gromov Theorem. This is explained in \cite{kl} and following Klartag's work, we present these two results here. One could enjoy the powerful localization method by observing how easily one obtains these isoperimetric inequalities. It is interesting to compare the following proofs with the original proofs in which more sophisticated ideas are used.

We start with the canonical sphere $\mathbb{S}^n$. The solution of the isoperimetric inequality on the sphere has been known for many years. Similar to the Euclidean counter-part, there are several different proofs for this result. The proof which uses localization is presented in \cite{gromil}. We recall this proof here:
\begin{theorem}[Isoperimetry on the Sphere] \label{isosphere}
Let $\mathbb{S}^n$ be the canonical sphere. Let $vol_n$ denotes the (canonical) Riemannian volume. For every open set $A$ and for every $\varepsilon>0$ we have:
\begin{eqnarray*}
vol_n(A+\varepsilon)\geq vol_n(B+\varepsilon),
\end{eqnarray*}
where $B$ is a spherical ball with the same volume as $A$.
\end{theorem}

\begin{proof}

Let $0<\kappa_1<1$ and $\varepsilon>0$ (small enough) be given such that $\kappa_1+\varepsilon< 1$. Let $\kappa_2=1-(\kappa_1+\varepsilon)$. We apply Theorem \ref{main} which gives us:
\begin{eqnarray*}
Sep(\mathbb{S}^n,\kappa_1,\kappa_2)\leq N_{\mathbb{S}^n}(\kappa_1,\kappa_2).
\end{eqnarray*}
Let us now investigate the right-hand side of the above inequality, \emph{i.e.} $N_{\mathbb{S}^n}(\kappa_1,\kappa_2)$. Since we are dealing with needles upon which the density of the probability measures are $\sin^{n-1}$-concave functions, proposition \ref{supra} explicitly gives: 
\begin{eqnarray*}
N_{\mathbb{S}^n}(\kappa_1,\kappa_2)=Sep(C\sin(t)^{n-1}dt, \kappa_1,\kappa_2),
\end{eqnarray*}
where $C$ is the normalizing constant.

On the other hand, we know that $N_{\mathbb{S}^n}(\kappa_1,\kappa_2)$ can be realized from two spherical balls $B_1$ and $B_2$ where $\mu(B_i)=\kappa_i$ with the balls being centered at opposite points. Indeed, the following formula holds for the \emph{measure} (normalized volume) of a ball of radius $r$ (centered at a point $x\in\mathbb{S}^n$) on the canonical sphere:
\begin{eqnarray*}
\frac{vol_n(B(x,r))}{vol_n(\mathbb{S}^n)}=\frac{\displaystyle\int_{0}^{r}\sin(t)^{n-1}dt}{\displaystyle\int_{0}^{\pi/2}\sin(t)^{n-1}dt}.
\end{eqnarray*}
We chose $r>0$ such that the right hand side in the equation above is equal to $\kappa_1$. The proof of the theorem follows by applying proposition \ref{impoo}.
\end{proof}

Here, we present the Levy-Gromov isoperimetric inequality:

\begin{theorem}[Levy-Gromov]
Let $M$ be a compact Riemannian manifold of dimension $n$ such that $Ricci(M)\geq Ricci(\mathbb{S}^n)$. We denote the Riemannian volume on $M$ and on $\mathbb{S}^n$ by $vol_n$. Let $A\subset M$ be an open subset of $M$ and let $B\subset \mathbb{S}^n$ be a ball on the sphere such that
\begin{eqnarray*}
\frac{vol_n(A)}{vol_n(M)}=\frac{vol_n(B)}{vol_n(\mathbb{S}^n)}.
\end{eqnarray*}
For every $\varepsilon>0$ we have:
\begin{eqnarray*}
\frac{vol_n(A+\varepsilon)}{vol_n(M)}\geq \frac{vol_n(B+\varepsilon)}{vol_n(\mathbb{S}^n)}.
\end{eqnarray*}
\end{theorem}
\begin{proof}
The proof is similar to the proof of Theorem \ref{isosphere}.

Let $\kappa_1,\kappa_2,\varepsilon$ be defined as in the proof of Theorem \ref{isosphere}. We apply Theorem \ref{main} which gives us:
\begin{eqnarray*}
Sep(M,\kappa_1,\kappa_2)\leq N_{M}(\kappa_1,\kappa_2).
\end{eqnarray*}
For any needle $\nu$ on $M$, the measure $\nu$ has density which is $\sin^{n-1}$-concave function. Moreover, the comparison of the diameter of $M$ and the one of $\mathbb{S}^n$ shows that the length of the support of $\nu$ is at most equal to $\pi$. Hence, we can use the result of proposition \ref{supra} and obtain that
\begin{eqnarray*}
N_{M}(\kappa_1,\kappa_2)\leq Sep(C\sin(t)^{n-1}dt, \kappa_1,\kappa_2),
\end{eqnarray*}
where $C$ is the normalizing constant.

The rest follows as in the proof of Theorem \ref{isosphere} and the proof of theorem follows.
\end{proof}

\begin{remark}
The power of Theorem \ref{main} lies in the fact that in order to obtain solution(s) to the isoperimetric problem, it is sufficient to find subsets which realize the \emph{optimal} needle (candidate) separation distance. For the case of the canonical sphere, this becomes straightforward thanks to proposition \ref{supra}. Other \emph{interesting} examples will follow in the next sections.
\end{remark}

\section{Sharp Isoperimetric Inequalities on Compact Rank One Symmetric Spaces}

The isoperimetric problem, even on a highly symmetric manifold, is usually very hard to solve. For instance, the isoperimetric inequality on (real) projective spaces was only solved up to dimension $3$ (see \cite{ros},\cite{rosr}). By applying our Theorem \ref{main}, we are able to solve the isoperimetric problem for the real, complex and quaternionic projective spaces in every dimension.

\begin{theorem}\label{realp}
Let $M$ be the canonical real projective space $\mathbb{R}P^n$. Let $vol_n$ denote the canonical Riemannian volume of $\mathbb{R}P^n$. Let $A$ be an open subset of $\mathbb{R}P^n$. For every $\varepsilon>0$ we have:
\begin{eqnarray*}
vol_n(A+\varepsilon)\geq vol_n(B+\varepsilon),
\end{eqnarray*}
where $B$ is 
\begin{itemize}
\item either an intrinsic ball with volume $=vol_n(A)$.
\item or a tube around a totally geodesic $\mathbb{R}P^k\subset\mathbb{R}P^n$ for a $1\leq k \leq n-1$ with volume $=vol_n(A)$.
\end{itemize}
\end{theorem}

\begin{theorem}\label{complexpro}
Let $M$ be the canonical complex projective space $\mathbb{C}P^n$. The (canonical) metric is the Fubini-Study metric with sectional curvature $ 1\leq Sec(M)\leq 4$. The volume element associated with the canonical Riemannian metric is denoted by $vol_{2n}$. Let $A$ be an open subset of $M$. For every $\varepsilon>0$ we have:
\begin{eqnarray*}
vol_{2n}(A+\varepsilon)\geq vol_{2n}(B+\varepsilon),
\end{eqnarray*}
where $B$ is either
\begin{itemize}
\item an intrinsic ball with volume equal to $vol_{2n}(A)$,
\item or a tube around a $\mathbb{C}P^k\subset\mathbb{C}P^n$ with volume equal to $vol_{2n}(A)$.
\end{itemize}
\end{theorem} 

\begin{theorem}\label{quaterpro}
Let $M$ be the canonical quaternionic projective space $\mathbb{H}P^n$. The (canonical) metric is the Fubini-Study metric with sectional curvature 
\begin{eqnarray*}
1\leq Sec(M)\leq 4.
\end{eqnarray*}
The volume element associated with the canonical Riemannian metric is denoted by $vol_{4n}$. Let $A$ be an open subset of $M$. For every $\varepsilon>0$ we have:
\begin{eqnarray*}
vol_{4n}(A+\varepsilon)\geq vol_{4n}(B+\varepsilon),
\end{eqnarray*}
where $B$ is either
\begin{itemize}
\item an intrinsic ball with volume equal to $vol_{4n}(A)$,
\item or a tube around a $\mathbb{H}P^k\subset \mathbb{H}P^n$ with volume equal to $vol_{4n}(A)$.
\end{itemize}
\end{theorem}

\begin{remark}
\begin{itemize}
\item In \cite{ros}, \cite{rosr}, the authors show that the isoperimetric solution on $\mathbb{R}P^3$ are given by geodesic balls or tubes around \emph{geodesics}. For different values of the volume, the isoperimetric solution (region) changes: for small or big volumes the solutions are the geodesic balls, while for the volumes around half of the total volume of $\mathbb{R}P^3$ the solution of the isoperimetric problem is given by a tube around $\mathbb{R}P^1$. 
\item As pointed out to us by a referee, in \cite{newref}, it is shown that for the complex projective space, some tubes around some $\mathbb{C}P^k\subset\mathbb{C}P^n$ which have the same volume of the geodesic ball in $\mathbb{C}P^n$ have smaller volume of their $\varepsilon$-neighborhood compared to the volume of $\varepsilon$-neighborhood of the geodesic ball. This already indicated the delicacy of the isoperimetric regions in the case of real, complex and quaternionic projective spaces. 
\end{itemize}
\end{remark}

\subsection{Strategy of Proof of Theorems \ref{realp}, \ref{complexpro} and \ref{quaterpro}}
The strategy of the proof is to first show that for every $\kappa_1,\kappa_2>0$ such that 
\begin{eqnarray*}
\kappa_1+\kappa_2<1,
\end{eqnarray*}
The following holds.
\begin{itemize}
\item For $\mathbb{R}P^n$, there exists a $k\in\{0,\cdots,n-1\}$ such that:
\begin{eqnarray*}
N_{\mathbb{R}P^n}(\kappa_1,\kappa_2)=Sep(C_k\sin(t)^{n-k-1}\cos(t)^k dt,\kappa_1,\kappa_2),
\end{eqnarray*}
where $C_k$ is the normalizing constant and the support is $(0,\pi/2)$.
\item For $\mathbb{C}P^n$, there exists a $k\in\{0,\cdots,n-1\}$ such that: 
\begin{eqnarray*}
N_{\mathbb{C}P^n}(\kappa_1,\kappa_2)=Sep(C_k\sin(t)^{2n-2k-1}\cos(t)^{2k+1}dt,\kappa_1,\kappa_2),
\end{eqnarray*}
where $C_k$ is the normalizing constant and the support is $(0,\pi/2)$.
\item For $\mathbb{H}P^n$, there exists a $k\in\{0,\cdots,n-1\}$ such that:
\begin{eqnarray*}
N_{\mathbb{H}P^n}(\kappa_1,\kappa_2)=Sep(C_k\sin(t)^{4n-4k-1}\cos(t)^{4k+3} dt,\kappa_1,\kappa_2),
\end{eqnarray*}
where $C_k$ is the normalizing constant and the support is $(0,\pi/2)$.
\end{itemize}

Once this is proven, it is sufficient to apply Proposition \ref{impoo} and find those subsets which realize the needle (candidate) separation distances.

The proof relies on the specific geometry of each of the above spaces in consideration. Therefore, we treat each space separately and present a proof for each case. We start with the real projective space, then treat the complex and quaternionic projective spaces.

\section{The proof of Theorem \ref{realp} for $\mathbb{R}P^n$}

The study of isoperimetric inequality on $\mathbb{R}P^n$ using the needles is different from the case of the round sphere even though the canonical $\mathbb{R}P^n$ is of constant sectional curvature $1$. Since minimizing geodesics on $\mathbb{R}P^n$ have length at most equal to $\pi/2$, (almost) the \emph{only} (but crucial) difference  for needles on $\mathbb{R}P^n$ from the needles defined on the canonical sphere $\mathbb{S}^n$ is the maximum length of their support. Hence, we are \emph{restricted} to study $\sin^{n-1}$-concave measures on intervals with maximum length $\pi/2$.

This difference has a big impact on the result when studying separation distance of needles in $\mathbb{R}P^n$. To get a feeling about this difference, we observe that on the sphere there is only \emph{one} needle with maximal support which is the needle $C\sin(t)^ndt$, however on the real projective space, there are several needles of maximal support: $C_k\sin(t)^k\cos(t)^{n-k}dt$. This fact plays a crucial role for the study of isoperimetric inequality of $\mathbb{R}P^n$ using the needle decomposition.

In order to prove Theorem \ref{realp}, we need two propositions, one purely analytic and the other with geometrical flavor. We start with the required analytic  result:

\begin{proposition} \label{proreal}
Let  $n>1$. Let
\begin{eqnarray*}
f(t)=C_1 f_1(t)f_2(t),
\end{eqnarray*}
where
\begin{eqnarray*}
f_1(t)&=& \cos(t+c_1)\cdots \cos(t+c_p) \\
f_2(t)&=& \cos(t+c_{p+1})\cdots \cos(t+c_{p+k}).
\end{eqnarray*}
Where $C_1>0$ and
\begin{eqnarray*}
p+k=n-1.
\end{eqnarray*}
For $i\in\{1,\cdots p\}$ we have $c_i\geq 0$, and for $j\in\{p+1\cdots,p+k\}$ we have $c_j<0$.
We assume $\mu_1=f(t)dt$ is a probability measure on $[0,L]$ where $L\leq \pi/2$ and where $f(L)=0$. Moreover, if $f(0)\neq 0$ then $k=0$.
Let 
\begin{eqnarray*}
g(t)=C_2\sin(t)^k\cos(t)^p.
\end{eqnarray*}
We assume $\mu_2=g(t)dt$ is a probability measure on $[0,\pi/2]$. 

Then for every $\kappa_1,\kappa_2>0$ such that
\begin{eqnarray*}
\kappa_1+\kappa_2<1,
\end{eqnarray*}
we have:
\begin{eqnarray*}
Sep(\mu_2,\kappa_1,\kappa_2)\geq Sep(\mu_1,\kappa_1,\kappa_2).
\end{eqnarray*}
\end{proposition}

\begin{proof}

When $k=0$, the proof is similar to the proof of Proposition \ref{supra}. We assume then $k\neq 0$ for which case the proof is more involved. 

We start by the following:
\begin{lemma} \label{firststep}
Let $t_1\in[0,L]$ (resp. $t_2\in[0,\pi/2]$) be the point where $f$ (resp. $g$) attains its maximum. Then 
\begin{eqnarray*}
t_1\leq t_2.
\end{eqnarray*}
\end{lemma} 
\begin{proof}
We have:
\begin{eqnarray*}
\frac{f(t)'}{f(t)}&=&-\sum_{i=1}^{p+k}\tan(t+c_i) \\
                      &\leq&-p\tan(t)+k\cot(t)\\
                      &=&\frac{g(t)'}{g(t)}.
\end{eqnarray*}

Since we have $f,g\geq 0$ and since $t_1$ is such that $f(t_1)'=0$ then from the above inequality we deduce that at point $t_1$, $g(t_1)'\geq 0$, which means in the interval $(0,t_1)$, the function $g$ is increasing. Therefore, we have:
\begin{eqnarray*}
t_1\leq t_2.
\end{eqnarray*}
This ends the proof of this Lemma.
\end{proof} 

Let $c$ be the maximum of $\{c_1,\cdots,c_p\}$.

We shift the function $f$ by $c$ to the \emph{right}, \emph{i.e.} the transformation $t\to t-c$. We define then 
\begin{eqnarray*}
F(t)=f(t-c),
\end{eqnarray*}
and we have:
\begin{eqnarray*}
F(t)=C_1\cos(t+d_1)\cdots\cos(t+d_{p+k}),
\end{eqnarray*}
where for every $i\in\{1,\cdots,p+k\}$ we have 
\begin{eqnarray*}
d_i=c_i-c.
\end{eqnarray*}
Clearly, $F(\pi/2)=0$ and the number of those $d_i<0$ is at least equal to $k$.

The function $F$ is the function $f$ shifted by $c$ in the direction of positive axis. Since we studied the function $f$ on the interval $[0,L]$, we study the function $F$ on $[c,\pi/2]$.

The following Lemma is the counterpart of Lemma \ref{firststep}:
\begin{lemma}\label{secondstep}
Let $t_1\in[c,\pi/2]$ (resp. $t_2\in[0,\pi/2]$) be the point where $F$ (resp. $g$) attains its maximum. Then 
\begin{eqnarray*}
t_1\geq t_2.
\end{eqnarray*}
\end{lemma} 
\begin{proof}
The Proof of this Lemma is similar to the proof of Lemma \ref{firststep}. We have:

\begin{eqnarray*}
\frac{F(t)'}{F(t)}&=&-\sum_{i=1}^{p+k}\tan(t+d_i) \\
                      &\geq&-p\tan(t)+k\cot(t)\\
                      &=&\frac{g(t)'}{g(t)}.
\end{eqnarray*}

Since on the interval $[c,\pi/2]$ we have $F\geq 0$ and since $t_1$ is such that $F(t_1)'=0$ then from the above inequality we deduce that $g(t_1)'\leq 0$. Therefore, on $(t_1,\pi/2)$ the function $g$ is decreasing. Hence we have:
\begin{eqnarray*}
t_1\geq t_2.
\end{eqnarray*}
This ends the proof of this Lemma.
\end{proof}

According to Lemma \ref{firststep}, we have $t_1\leq t_2$. We translate the function $f$ by $t_2-t_1$ to the right, \emph{i.e.} we perform the transformation $t\to t-(t_2-t_1)$. Let the translated function be denoted by $F_m$. According to Lemma \ref{secondstep}, we have 
\begin{eqnarray*}
[t_2-t_1,L+(t_2-t_1)]\subset [0,\pi/2].
\end{eqnarray*}
Moreover, the maximum points of $F_m$ and $g$ coincide. Recall that $Fdt$ is a probability measure on $[c,\pi/2]$ hence $\mu_m=F_mdt$ is a probability measure on 
\begin{eqnarray*} 
[t_2-t_2,L+(t_2-t_2)].
\end{eqnarray*}
Moreover, both $F_m$ and $g$ are increasing from $(t_2-t_1,t_2)$ and decreasing on $(t_2,L+(t_2-t_1))$. Hence we have:
\begin{eqnarray*}
F_m(t_2)\geq g(t_2).
\end{eqnarray*}
Indeed otherwise, the function $F_m$ would be strictly smaller than $g$ which is not possible since both $g(t)dt$ and $F_m(t)dt$ are probability measures.

We are set to finalize the proof of proposition \ref{proreal}.

Let $\kappa_1,\kappa_2>0$ with 
\begin{eqnarray*}
\kappa_1+\kappa_2<1,
\end{eqnarray*}
be given. Recall $\mu_1=f(t)dt$.

It is clear that:
\begin{eqnarray*}
Sep((\mu_1,[0,L]),\kappa_1,\kappa_2)=Sep((\mu_{m},[t_2-t_1,L+(t_2-t_1)]),\kappa_1,\kappa_2).
\end{eqnarray*}
Indeed, assume the sets $U_1,U_2\subset [0,T]$ where
\begin{eqnarray*}
\mu_1(U_i)=\kappa_i,
\end{eqnarray*}
for $i=1,2$ realize the separation distance $Sep(\mu_1,\kappa_1,\kappa_2)$ which we know from Lemma \ref{comp} that are connected intervals. Then for $i=1,2$, the sets 
\begin{eqnarray*}
W_i=U_i+(t_2-t_1)
\end{eqnarray*}
realize the separation distance $Sep(\mu_{m},\kappa_1,\kappa_2)$.


For $t\in (t_2-t_1,L+(t_2-t_1))$, consider the function:
\begin{eqnarray*}
q(t)=\mu_2((0,t))-\mu_{m}((t_2-t_1,t)).
\end{eqnarray*}
It is clear that $q$ is a continuous function and it is clear that for $t$ small enough we have 
\begin{eqnarray*}
q(t)>0,
\end{eqnarray*}
and for $t$ large enough we have:
\begin{eqnarray*}
q(t)<0.
\end{eqnarray*}
Indeed, for $t$ small enough we have $g(t)\geq F_m(t)$ hence $\mu_2(0,t)\geq \mu_m(t_2-t_1,t)$. And for $t=L+(t_2-t_1)$, we have $\mu_{m}((t_2-t_1,t))=1$ and $\mu_2((0,t))<1$.

Therefore, there exists a $t_0\in (t_2-t_1,L+(t_2-t_1))$ such that:
\begin{eqnarray*}
q(t_0)=0.
\end{eqnarray*}
The above translates to the fact that:
\begin{eqnarray*}
\mu_{m}(t_2-t_1,t_0)&=&\mu_2(0,t_0)\\
              &=& \kappa,
\end{eqnarray*}

for some $\kappa\in(0,1)$.

We have different possibilities for the values of $\kappa_1$ and $\kappa_2$ summarized as follows and treated separately:
\begin{itemize}
\item We have $\kappa_1\leq \kappa$ and $\kappa_2\leq 1-\kappa$.
\item We have $\kappa_1\leq \kappa$ and $\kappa_2\geq 1-\kappa$.
\item We have $\kappa_1\geq \kappa$ and $\kappa_2\leq 1-\kappa$.
\end{itemize}

For the first case, where $\kappa_1\leq \kappa$ and $\kappa_2\leq 1-\kappa$, by the properties of the function $q$ we have:
\begin{eqnarray*}
\mu_m(W_1)&\leq& \mu_2(W_1) \\
\mu_m(W_2)&\leq& \mu_2(W_2).
\end{eqnarray*}
Therefore, for $F_1,F_2\subset [0,\pi/2]$ such that:
\begin{eqnarray*}
\mu_2(F_1)&=&\kappa_1 \\
\mu_2(F_2)&=&\kappa_2,
\end{eqnarray*}
we have
\begin{eqnarray*}
d(F_1,F_2)\geq d(W_1,W_2),
\end{eqnarray*}
where we refer to definition \ref{defneedle} for the definition of the distance $d(.,.)$ between the intervals.

We now treat the second case, where $\kappa_1\leq \kappa$ and $\kappa_2\geq 1-\kappa$. 

Let $F_1=(0,u_1)\subset (0,\pi/2)$ be the interval for which we have
\begin{eqnarray*}
\mu_2(F_1)=\kappa_1.
\end{eqnarray*}
Let $W_1=(t_2-t_1,w_1)$ such that:
\begin{eqnarray*}
\mu_m(W_1)=\kappa_1.
\end{eqnarray*}

Since $\kappa_1\leq \kappa$, we have:
\begin{eqnarray*}
u_1\leq w_1.
\end{eqnarray*}
Let
\begin{eqnarray*}
W_2&=&(w_2,L+(t_2-t_1)) \\
F_2&=&(u_2,\pi/2),
\end{eqnarray*}
such that, we have:
\begin{eqnarray*}
\mu_2(F_2)&=& \mu_m(W_2)\\
                  &=& \kappa_2.
\end{eqnarray*}


Suppose the following inequality holds:
\begin{equation} \label{funny}
\vert u_1-w_1\vert\leq \vert u_2-w_2\vert.
\end{equation}

It is clear that we have:
\begin{equation} \label{ira}
\mu_2(u_1,u_2)=\mu_m(w_1,w_2).
\end{equation}
From inequality (\ref{funny}), we assumed the length of the interval $(u_1,u_2)$ be larger than the length of the interval $(w_1,w_2)$. Further, the function $F_m$ is larger than $g$ on the interval $(w_1,w_2)$ which contradicts equation (\ref{ira}). Therefore, we have:
\begin{eqnarray*}
d(F_1,F_2)\geq d(W_1,W_2).
\end{eqnarray*}

We treat now the last case for $\kappa_1,\kappa_2$ where we have $\kappa_1\geq \kappa$ and $\kappa_2\leq 1-\kappa$. The proof is similar to the previous case. 

Let $F_1=(0,u_1)\subset (0,\pi/2)$ be the interval for which we have
\begin{eqnarray*}
\mu_2(F_1)=\kappa_1.
\end{eqnarray*}
Let $W_1=(t_2-t_1,w_1)$ such that:
\begin{eqnarray*}
\mu_m(W_1)=\kappa_1.
\end{eqnarray*}

Since $\kappa_1\geq \kappa$, we have:
\begin{eqnarray*}
u_1\geq w_1.
\end{eqnarray*}
Let
\begin{eqnarray*}
W_2&=&(t_2-t_1,w_2) \\
F_2&=&(0,u_2),
\end{eqnarray*}
such that we have:
\begin{eqnarray*}
\mu_2(F_2)&=& \mu_m(W_2)\\
                  &=& 1-\kappa_2.
\end{eqnarray*}

Suppose we have:
\begin{eqnarray*}
\vert u_1-w_1\vert\geq \vert u_2-w_2\vert,
\end{eqnarray*}

It is clear that we have:
\begin{equation} \label{iraa}
\mu_2(u_1,u_2)=\mu_m(w_1,w_2).
\end{equation}
We assumed the length of the interval $(u_1,u_2)$ be larger than the length of the interval $(w_1,w_2)$. Further, the function $F_m$ is larger than $g$ on the interval $(w_1,w_2)$. Therefore, it contradicts equation (\ref{iraa}). Therefore, we have:
\begin{eqnarray*}
d(F_1,F_2)\geq d(W_1,W_2).
\end{eqnarray*}

We have demonstrated that for every $\kappa_1, \kappa_2>0$ such that
\begin{eqnarray*}
\kappa_1+\kappa_2<1,
\end{eqnarray*}
we have
\begin{equation} \label{pfu}
d(F_1,F_2)\geq d(W_1,W_2).
\end{equation}
According to Lemma \ref{comp} we know that the separation distance is realized by connected intervals, and using the inequality (\ref{pfu}) we deduce that:
\begin{eqnarray*}
Sep(\mu_2,\kappa_1,\kappa_2) &\geq& d(F_1,F_2) \\ 
                             &\geq& d(W_1,W_2)\\
                             &=&   Sep(\mu_{m},\kappa_1,\kappa_2).
\end{eqnarray*}

The proof of Proposition \ref{proreal} is completed.

\end{proof}

\begin{remark}
In the proof of Proposition \ref{proreal}, we showed that the support of the measure $\mu_{m}$ is strictly included in the support of the measure $\mu_2$ which is the interval $(0,\pi/2)$. Then it was easy to demonstrate that the separation distance of $\mu_2$ is larger than the separation distance of the measure $\mu_{m}$. The smaller the support of a measure is, the larger the maximum point would be and the measure is \emph{more} concentrated around its maximum point which makes the separation distance less than a measure with larger support. This idea was already known and is explained in \cite{grobook} (pages $154-155$).
\end{remark}

Having proved the analytic result of proposition (\ref{proreal}), we present a result which has to do with the geometry of $\mathbb{R}P^n$:
\begin{proposition} \label{candreal}
Let $n\geq 1$. Let $\nu$ be a needle candidate on $\mathbb{R}P^n$. Let $f$ be the density of the pull-back of $\nu$ on $A=(0,L)$ with $L\leq \pi/2$. Then
\begin{eqnarray*}
f(t)= C\:\Pi_{1}^{n-1}f_i(t),
\end{eqnarray*}
where $C\in\mathbb{R}$ and for every $i\in\{1,\cdots,n-1\}$ there exists $\beta_i\in\mathbb{R}$ such that
\begin{eqnarray*}
f_i(t)=\sin(t+\beta_i).
\end{eqnarray*}
\end{proposition}

\begin{proof}
Let the dimension $n\geq 1$ of $\mathbb{R}P^n$ be fixed. Based on definition \ref{nc} of the needle candidates we know there exists a family of Jacobi fields $J_1(t),\cdots J_{n-1}(t)$ along $\gamma(t)$ for $t\in A$. As in the definition \ref{nc}, we let $J_{n}=\gamma'$. We assume this family consists of linearly independent (tangent) vectors. Otherwise there is nothing to prove. 

There exists a constant $c>0$ and the density of the pull-back of the needle $\nu$, denoted by $f$, equals: 
\begin{eqnarray*}
f(t)= c\sqrt{\det{(<J_i(t),J_k(t)>)_{i,k=1,\cdots, n}}}.
\end{eqnarray*}

The Riemannian volume $vol_{n}$ of $\mathbb{R}P^n$ can be transported by the (inverse of the) exponential map on the tangent space at any point. By Linear algebra, we know that
\begin{eqnarray*}
\sqrt{\det{(<J_i(t),J_k(t)>)_{i,k=1,\cdots, n}}},
\end{eqnarray*}
is the determinant of the \emph{Gram} matrix associated to vectors $J_i(t)\in T_{\gamma(t)}\mathbb{R}P^n$ for $i=1,\cdots,n$. This quantity is equal to the  (Riemannian) volume of the parallelepiped generated by the linearly independent tangent vectors $\{J_i(t)\}_{i=1}^{n}$ in the tangent space $T_{\gamma(t)}\mathbb{R}P^n$. We denote by:
\begin{eqnarray*}
[J_1,\cdots,J_{n}],
\end{eqnarray*}
the parallelepiped, generated by the vectors $J_i$ ($i=1,\cdots,n$). For every $t\in A$, we denote (abusively) by $vol_{n}$ the Riemannian volume in the tangent space $T_{\gamma(t)}\mathbb{R}P^n$. By definition \ref{nc}, there exists a point $t\in A$ on which we have:
\begin{eqnarray*}
<J_i(t),\gamma'(t)>\: =\: <J'_i(t),\gamma'(t)>\: = 0
\end{eqnarray*}
for $i=1,\cdots, n-1$ and
\begin{eqnarray*}
<J'_i(t),J_k(t)>\:=\:<J'_k(t),J_i(t)>,
\end{eqnarray*}
for $i,k=1,\cdots,n-1$. An easy exercise concerning the Jacobi fields (see \cite{sakai},\cite{gallot}) asserts that if such equality holds for one point, then it holds on every point on $\gamma(t)$. Thus, the \emph{frame} of Jacobi fields $J_i$ are orthogonal to $\gamma'$ for every $t\in A$. Furthermore, without loss of generality we may assume that $\{J_i(t)\}$ form an orthogonal basis of the tangent space at $\gamma(t)$. Indeed, since for $i=1,\cdots,n-1$, the vectors  $\{J_i(t)\}$ are linearly independent, there exists a constant $(n-1)\times(n-1)$ matrix $A$ (along $\gamma(t)$ such that $\{AJ_i(t)\}$ are linearly independent and orthogonal Jacobi fields. 

We have the following result in Riemannian geometry:
\begin{proposition} \label{kr}
Let $M$ be a Riemannian manifold of dimension $n$ of constant sectional curvature $1$. Let $\sigma$ be a geodesic segment in $M$. Let $\{E_1,\cdots,E_n\}$ be orthogonal and parallel along $\sigma$ with $E_1$ be the unit tangent vector along $\sigma$. Then the general solution of a Jacobi field along $\sigma$ is given by:
\begin{eqnarray*}
\sum_i (a_i\sin(t)+b_i\cos(t))E_i.
\end{eqnarray*}
\end{proposition}


Therefore, by applying proposition \ref{kr} to the orthogonal Jacobi fields $\{AJ_i(t)\}$ the proof follows.

\end{proof}

\subsection{End Proof of Theorem \ref{realp}}

We quickly remind some known facts on Riemannian geometry:
\begin{definition} \label{focal}
Let $M$ be a closed Riemannian manifold of dimension $n$ and let $H\subset M$ be a submanifold of dimension $0<k<n$. Let $c:[a,b]\to M$ be a geodesic in $M$ such that:
\begin{itemize}
\item The geodesic $c$ is orthogonal to $H$.
\item $c(a)\in H$.
\item $c(a)'\in T_{c(a)}H^{\perp}$.
\end{itemize}
Then, $c(t_0)\in c(a,b]$ is called a \emph{focal} point of $H$ if there exists a non-trivial Jacobi field $J$ along $c$ with $J(a)\in T_{c(a)}H$ such that $J(t_0)=0$. Furthermore, the distance between $c(a)$ to $c(t_0)$ is called the \emph{focal} distance of $c(t_0)$ to $H$.
\end{definition}

We invite the reader to consult \cite{wilh} and \cite{herm} for some related topics on focal distances of submanifolds in $\mathbb{R}P^n$, $\mathbb{C}P^n$ or $\mathbb{H}P^n$. For our purpose, we need to know that the maximal focal distance of any submanifold in $\mathbb{R}P^n$ is equal to $\pi/2$.

Therefore, according to Lemma \ref{comp} and particularly inequality (\ref{eqn:big}), for the purpose of estimating the separation distance, it is enough to consider those needle candidates $\nu$ for which the density of their pull-back is defined on an interval $(0,L)$ with $L\leq \pi/2$ and is equal to $f(t)$ as it is defined in Proposition \ref{proreal}. Furthermore, according to this same proposition \ref{proreal} we deduce that needle candidates of the form 
\begin{eqnarray*}
g(t)=\sin(t)^{m_1}\cos(t)^{m_2},
\end{eqnarray*}
with $m_1,m_2\geq 0$ and $m_1+m_2=n-1$ maximize the needle separation distance.

Therefore, for $\kappa_1,\kappa_2>0$ such that
\begin{eqnarray*}
\kappa_1+\kappa_2<1,
\end{eqnarray*}
there exists a $k\in\{1\cdots,n\}$ such that:
\begin{eqnarray*}
N_{\mathbb{R}P^n}(\kappa_1,\kappa_2)=Sep(C_k\sin(t)^{n-k}\cos(t)^{k-1}dt,\kappa_1,\kappa_2),
\end{eqnarray*}

Therefore, according to Proposition \ref{impoo}, it remains to find those subsets of $\mathbb{R}P^n$ which realize the above (optimal) needle separation distance.

The \emph{measure} (normalized volume) of balls in $\mathbb{R}P^n$ is the same as the \emph{measure} of balls in $\mathbb{S}^n$. Indeed, for $0<r\leq\pi/2$, and $x\in\mathbb{R}P^n$:
\begin{eqnarray*}
\frac{vol_n(B(x,r))}{vol_n(\mathbb{R}P^n)}=\frac{\displaystyle\int_{0}^{r}\sin(t)^{n-1}dt}{\displaystyle\int_{0}^{\pi/2}\sin(t)^{n-1}dt}.
\end{eqnarray*}
The \emph{measure} (normalized volume) of a tube around a totally geodesic
\begin{eqnarray*}
\mathbb{R}P^{n-k}\subset\mathbb{R}P^n
\end{eqnarray*} 
is given by:
\begin{eqnarray*}
\frac{vol_n(\mathbb{R}P^{n-k}+r)}{vol_n(\mathbb{R}P^n)}=\frac{\displaystyle\int_{0}^{r}\sin(t)^{k-1}\cos(t)^{n-k}dt}{\displaystyle\int_{0}^{\pi/2}\sin(t)^{k-1}\cos(t)^{n-k}dt}.
\end{eqnarray*}
One could consult \cite{grey} for a proof of the above formula.

We apply the result of Proposition \ref{impoo} and Lemma \ref{proreal} and deduce that for every open set $U\subset\mathbb{R}P^n$ we have:
\begin{eqnarray*}
vol_n(U+\delta)\geq vol_n(B+\delta),
\end{eqnarray*}
where $B$ is either a ball or a tube around some $\mathbb{R}P^{k}$ where:
\begin{eqnarray*}
vol_n(U)=vol_n(B).
\end{eqnarray*}
This ends the proof of Theorem \ref{realp}.

\section{The proof of Theorem \ref{complexpro} for $\mathbb{C}P^n$}

We inherit $\mathbb{C}P^n$ with the canonical Riemannian structure and denote $vol_{2n}$ by the Riemannian volume. 

We start with an analytic result:

\begin{proposition} \label{procomp}
Let $n\geq 1$. Let
\begin{eqnarray*}
f(t)=C_1 f_1(t)f_2(t)f_3(t),
\end{eqnarray*}
where
\begin{eqnarray*}
f_1(t)&=&\cos(t+a_1)\cdots\cos(t+a_p)\\ 
f_2(t)&=&\cos(t+b_1)\cdots\cos(t+b_q)\\
f_3(t)&=&\sin(2t+c_1)\cdots \sin(2t+c_s).
\end{eqnarray*}
where $C_1>0$, $s\geq 1$ and
\begin{eqnarray*}
p+q+s=2n-1.
\end{eqnarray*}
Moreover for $i\in\{1,\cdots\ \max\{p,q,s\}\}$ we have $a_i, c_i\geq 0$ and $b_i<0$. We assume $\mu_1=f(t)dt$ is a probability measure on $[0,L]$ where $L\leq \pi/2$ and where 
\begin{eqnarray*}
f(0)=f(L)=0.
\end{eqnarray*}
Let 
\begin{eqnarray*}
g(t)=C_2\sin(t)^{m_1}\cos(t)^{m_2}\sin(2t)^{m_3},
\end{eqnarray*}
where 
\begin{eqnarray*}
m_1+m_2+m_3=2n-1,
\end{eqnarray*}
and 
\begin{eqnarray*}
m_2&\leq& p\\
q&\leq& m_1\\
s&\leq& m_3.
\end{eqnarray*}
Moreover we assume $\mu_2=g(t)dt$ is a probability measure on $[0,\pi/2]$. 

Then for every $\kappa_1,\kappa_2>0$ where
\begin{eqnarray*}
\kappa_1+\kappa_2<1,
\end{eqnarray*}
we have:
\begin{eqnarray*}
Sep(\mu_2,\kappa_1,\kappa_2)\geq Sep(\mu_1,\kappa_1,\kappa_2).
\end{eqnarray*}
\end{proposition}

\begin{proof}

The strategy of proof is similar to the proof of proposition \ref{proreal}.

We start by the following:
\begin{lemma} \label{compfirststep}
Let $t_1\in[0,L]$ (resp. $t_2\in[0,\pi/2]$) be the point where $f$ (resp. $g$) attains its maximum. Then 
\begin{eqnarray*}
t_1\leq t_2.
\end{eqnarray*}
\end{lemma} 
\begin{proof}
On the interval $[0,L]$ we have:
\begin{eqnarray*}
\frac{f(t)'}{f(t)}&=&-\sum_{i=1}^{p}\tan(t+a_i)-\sum_{i=1}^{q}\tan(t+b_i)+2\sum_{i=1}^{s}\cot(t+c_i)\\
                  &\leq&-m_2\tan(t)+m_1\cot(t)+m_3\cot(2t)\\
                  &=&\frac{g(t)'}{g(t)}.
\end{eqnarray*}


Since we have $f,g\geq 0$ and since $t_1$ is such that $f(t_1)'=0$ then from the above inequality we deduce that at point $t_1$, $g(t_1)'\geq 0$. Hence the function $g$ is increasing on the interval $(0,t_1)$. Therefore, we have:
\begin{eqnarray*}
t_1\leq t_2.
\end{eqnarray*}
This ends the proof of this Lemma.
\end{proof} 

Let $c$ be the maximum of $\{a_1,\cdots,a_p,\frac{c_1}{2}\cdots,\frac{c_s}{2}\}$.

We shift the function $f$ by $c$ \emph{i.e.} the transformation $t\to t-c$. We define then 
$F(t)=f(t-c)$ and we have:
\begin{eqnarray*}
F(t)=C_1 f_1(t-c)f_2(t-c)f_3(t-c),
\end{eqnarray*}

and we observe that $F(\pi/2)=0$.

The function $F$ is the function $f$ shifted by $c$ in the direction of positive axis. Since we studied the function $f$ on the interval $[0,L]$, we study the function $F$ on $[c,\pi/2]$.

The following Lemma is the counterpart of Lemma \ref{compfirststep}:
\begin{lemma}\label{compsecondstep}
Let $t_1\in[c,\pi/2]$ (resp. $t_2\in[0,\pi/2]$) be the point where $F$ (resp. $g$) attains its maximum. Then 
\begin{eqnarray*}
t_1\geq t_2.
\end{eqnarray*}
\end{lemma} 
\begin{proof}

We make the change of variable $t\to\pi/2-t$ and observe that:
\begin{eqnarray*}
g(\pi/2-t)=C_2\cos(t)^{m_1}\sin(t)^{m_2}\sin(2t)^{m_3},
\end{eqnarray*}
and 
\begin{eqnarray*}
F(\pi/2-t)&=&f(\pi/2-t-c)\\
          &=&C_1f_1(\pi/2-t-c)f_2(\pi/2-t-c)f_3(\pi/2-t-c),
\end{eqnarray*}
where 
\begin{eqnarray*}
f_1(\pi/2-t-c)&=&\cos(\pi/2-t+a_1-c)\cdots\cos(\pi/2-t+a_p-c)\\ 
f_2(\pi/2-t-c)&=&\cos(\pi/2-t+b_1-c)\cdots\cos(\pi/2-t+b_q-c)\\
f_3(\pi/2-t-c)&=&\sin(2t+c_1-2c)\cdots \sin(2t+c_s-2c).
\end{eqnarray*}

Since $a_i, c_i\geq 0$ and $b_i<0$ and $c$ is the maximum of $\{a_1,\cdots,a_p,c_1\cdots,c_s\}$, we have:
\begin{eqnarray*}
f_1(\pi/2-t-c)&=&\cos(t+d_1)\cdots\cos(t+d_p)\\ 
f_2(\pi/2-t-c)&=&\cos(t+h_1)\cdots\cos(t+h_q)\\
f_3(\pi/2-t-c)&=&\sin(2t+c_1-2c)\cdots \sin(2t+c_s-2c),
\end{eqnarray*}
where $d_i\leq 0$ and $h_i\geq 0$. Therefore, we apply Lemma \ref{compfirststep} to $F(\pi/2-t)$ and $g(\pi/2-t)$ and the proof follows.

\end{proof}

According to Lemmas (\ref{compfirststep}) and (\ref{compsecondstep}), the function $f(t-(t_2-t_1))$ obtains its maximum at the point $t_2$ which coincides with the point where $g$ attains its maximum. Moreover, we have:
\begin{eqnarray*}
[t_2-t_1,L+(t_2-t_1)]\subset [0,\pi/2].
\end{eqnarray*}
Hence, the support of $f(t-(t_2-t_1))$ is strictly contained in the support of $g$ and therefore:
\begin{eqnarray*}
f(t_2)\geq g(t_1).
\end{eqnarray*}
The rest of the proof is exactly similar to the proof of proposition (\ref{proreal}). Therefore:
\begin{eqnarray*}
Sep(\mu_2,\kappa_1,\kappa_2)\geq Sep(\mu_1,\kappa_1,\kappa_2).
\end{eqnarray*}

\end{proof}

\begin{remark}
In Proposition \ref{procomp} we could have replaced the function $f_1(t)f_2(t)f_3(t)$ by suitable functions of the form: 
\begin{eqnarray*}
\Pi_{i=1}^{p}\sin(\alpha_i t+\beta_i),
\end{eqnarray*}
where $\alpha_i\in[1,2]$ and $\beta_i\in\mathbb{R}$ such that $f_2(0)=f_2(L)=0$. The proof is indeed similar.
\end{remark}

Next result concerns the needle candidates on $\mathbb{C}P^n$ and the result will be decisive for the proof of the isoperimetric inequality on $\mathbb{C}P^n$.

\begin{proposition} \label{candcomp}
Let $n\geq 1$. Let $\nu$ be a needle candidate on $\mathbb{C}P^n$. Let $f$ be the density function of the pull-back of $\nu$ on an interval $A=(0,L)\subset\mathbb{R}$ such that $L\leq \pi/2$. Then 
\begin{eqnarray*}
f(t)=\Pi_{i=1}^{2n-1}f_i(t),
\end{eqnarray*}
where
for every $i\in\{1,\cdots,2n-1\}$ there exists $\alpha_i\in[1,2]$ and $\beta_i\in\mathbb{R}$ such that
\begin{eqnarray*}
f_i(t)=\sin(\alpha_i t+\beta_i).
\end{eqnarray*}
\end{proposition}

\begin{proof}

Let $\nu$ be a needle candidate supported on a geodesic $\gamma\subset \mathbb{C}P^n$ which is parametrized by its arclength $t\in A\subseteq [0,\pi/2]$. Let $f$ be the density function of the pull-back of $\nu$.

All we know about $f$ at this stage is that $f$ is a $\sin^{2n-1}$-concave function. We will dig more and find more properties for $f$ which has to do with the geometry of $\mathbb{C}P^n$. 

\begin{lemma} \label{dig}
Let the interval $(0,L)$ be as above. Let $f$ be the density of the pull-back of the needle candidate $\nu$. Then
\begin{eqnarray*}
f=f_1 f_2
\end{eqnarray*}
where $f_1$ is a function satisfying the inequality:
\begin{eqnarray*}
f_1''+4f_1\leq 0,
\end{eqnarray*}
and $f_2$ is a $\sin^{2n-2}$-concave function.
\end{lemma}

\begin{proof}

The complex projective space has real dimension $2n$. Based on definition \ref{nc} of the needle candidates we know there exists a family of Jacobi fields $J_1(t),\cdots J_{2n-1}(t)$ along $\gamma$ for $t\in A$. Let $J_{2n}=\gamma'$ and assume this family consists of linearly independent (tangent) vectors. 

There exists a constant $c>0$ and the density of the needle $\nu$, denoted by $f$, equals: 
\begin{eqnarray*}
f(t)= c\sqrt{\det{(<J_i(t),J_k(t)>)_{i,k=1,\cdots, 2n}}}.
\end{eqnarray*}

The Riemannian volume $vol_{2n}$ of $\mathbb{C}P^n$ can be transported by the (inverse of the) exponential map on the tangent space at any point. By Linear algebra, we know that
\begin{eqnarray*}
\sqrt{\det{(<J_i(t),J_k(t)>)_{i,k=1,\cdots, 2n}}},
\end{eqnarray*}
is the determinant of the \emph{Gram} matrix associated to vectors $J_i(t)\in T_{\gamma(t)}\mathbb{C}P^n$ for $i=1,\cdots,2n$. This quantity is equal to the  (Riemannian) volume of the parallelepiped generated by the linearly independent tangent vectors $\{J_i(t)\}_{i=1}^{2n}$ in the tangent space $T_{\gamma(t)}\mathbb{C}P^n$. We denote by:
\begin{eqnarray*}
[J_1,\cdots,J_{2n}],
\end{eqnarray*}
the parallelepiped, generated by the vectors $J_i$ ($i=1,\cdots,2n$). For every $t\in A$ we denote (abusively) by $vol_{2n}$ the Riemannian volume in the tangent space $T_{\gamma(t)}\mathbb{C}P^n$. Here, the space of Jacobi fields along a geodesic $\gamma$ is a linear vector space of dimension $4n$. By definition \ref{nc}, there exists a point $t\in A$ on which we have:
\begin{eqnarray*}
<J_i(t),\gamma'(t)>\: =\: <J'_i(t),\gamma'(t)>\:=\: 0
\end{eqnarray*}
for $i=1,\cdots, 2n-1$ and
\begin{eqnarray*}
<J'_i(t),J_k(t)>\:=\:<J'_k(t),J_i(t)>,
\end{eqnarray*}
for $i,k=1,\cdots,2n-1$. An easy exercise concerning the Jacobi fields (see \cite{sakai},\cite{gallot}) asserts that if such equality holds for one point, then it holds on every point on $\gamma(t)$. Thus, the \emph{frame} of Jacobi fields $J_i$ are orthogonal to $\gamma'$ for every $t\in A$.

For every $i=1,\cdots,2n$, let
\begin{eqnarray*}
J_i(t)=f_i(t)E_i(t),
\end{eqnarray*}
where $\Vert E_i(t)\Vert=1$.

With the same argument presented in the proof of proposition (\ref{candreal}) we assume the family $\{E_i(t)\}$ forming an orthonormal basis of the tangent space at the point $\gamma(t)$ (and hence on every point along $\gamma(t)$). 

Let $t\in A$, and let $\gamma(t)$ be a point on the geodesic $\gamma$. Let $T_{\gamma(t)}\mathbb{C}P^n$ be the tangent space at this point. By linear algebra we know that:
\begin{eqnarray*}
vol_{2n}([J_1(t),\cdots,J_{2n}(t)])&=&\Vert J_{2n}(t)\Vert vol_{2n-1}([J_1(t),\cdots,J_{2n-1}]) \\
                                   &=& vol_{2n-1}([J_1(t),\cdots,J_{2n-1}]).
\end{eqnarray*}
Indeed, the geodesic $\gamma$ is parametrized by its arclength, hence
\begin{eqnarray*}
\Vert J_{2n}\Vert&=&\Vert\gamma'\Vert \\
                 &=& 1.
\end{eqnarray*}
Moreover,
\begin{eqnarray*}
[J_1(t),\cdots,J_{2n-1}],
\end{eqnarray*}
is the parallelepiped generated by $J_i(t)$ for $i=1,\cdots,2n-1$ and is contained in a $(2n-1)$-dimensional vector subspace of $T_{\gamma(t)}\mathbb{C}P^n$ (which is  orthogonal $\gamma'(t)$) . This parallelepiped is in fact a $(2n-1)$-dimensional rectangle as $\{J_i(t)\}_{i=1}^{2n}$ form an orthogonal basis of the tangent space $T_{\gamma(t)}\mathbb{C}P^n$.

At the point $\gamma(t)$, consider the orthonormal basis:
\begin{eqnarray*}
e_1,e_2,\cdots,e_{2n},
\end{eqnarray*}
where  for a $m\in\{1,\cdots,2n\}$, we have:
\begin{eqnarray*}
e_{m}=ie_{m-1},
\end{eqnarray*}
where $ie_{m-1}$ is the complex multiplication of $e_{m-1}$ by the complex number $i$. Therefore, the vectors $e_{m-1}$ and $e_m$ form an orthonormal base of $\mathbb{C}^2\subset T_{\gamma(t)}\mathbb{C}P^n$.



Again, by linear algebra, we have:
\begin{eqnarray*}
vol_{2n-1}([J_1(t),\cdots J_{2n-1}(t)]=\Vert J_m(t)\Vert vol_{2n-2}(R_{2n-2}(t)),
\end{eqnarray*}
where
\begin{eqnarray*}
R_{2n-2}(t)=[J_2(t),\cdots, J_m(t)^{\circ},\cdots, J_{2n-1}(t)],
\end{eqnarray*}
is the $(2n-2)$-dimensional rectangle on which $J_m$ is orthogonal. 

By the property of $\mathbb{C}P^n$, the sectional curvature of the $2$-plane $\{e_m,ie_m\}$ is equal to $4$. Since $\mathbb{C}P^n$ is a symmetric space, for every $t\in A$ the Jacobi field $J_m$, remains parallel along $\gamma(t)$ . This is a strong property which relies on the symmetry of $\mathbb{C}P^n$ and can be verified in more detail in \cite{gallot} and \cite{sakai}. 

Hence, by definition, $J_m$ is a Jacobi field on a $\mathbb{S}^2(4)$ which is a sphere of constant curvature $4$ and thus,  
 \begin{eqnarray*}
f_1(t)=\Vert J_m(t)\Vert,
\end{eqnarray*}
satisfies the following differential inequality:
\begin{equation} \label{eqn:closer}
f_1''+4f_1\leq 0.
\end{equation}

Let
\begin{eqnarray*}
f_2(t)=vol_{2n-2}([J_2(t),\cdots,J_m^{\circ},\cdots, J_{2n-1}(t)]).
\end{eqnarray*}

Since the norm of the Jacobi fields $J_i(t)$ for every $i\in\{2,\cdots,2n-2\}$ is a $\sin$-concave function for $t\in A$, therefore according to lemma \ref{comp}, the function $f_2(t)$ is a $\sin^{2n-2}$-concave function. This ends the proof of the lemma.

\end{proof}

Let 
\begin{eqnarray*}
p:\mathbb{S}^{2n+1}\to \mathbb{C}P^n
\end{eqnarray*}
be the Riemannian submersion from the canonical $2n+1$-dimensional sphere to the $n$-dimensional complex projective space.

Let $\gamma\subset \mathbb{C}P^n$ be a (minimal) geodesic on the complex projective space. It is known (see for instance \cite{sakai}) that for every  Jacobi field  $J$ along $\gamma$ , there exists a geodesic $\hat{\gamma}$ on $\mathbb{S}^{2n+1}$ and a Jacobi field along this geodesic which projects on $\gamma$ and $J$. According to Proposition (\ref{kr}) we know the general form of orthogonal Jacobi fields along a geodesic on $\mathbb{S}^{2n+1}$. Therefore, by projection on $\mathbb{C}P^n$ and applying lemma \ref{dig} the proof of the proposition is completed.

\end{proof}


For our purpose, we could go further and only consider those needle candidates $\nu$ such that the density of their pull-back is given by:
\begin{eqnarray*}
C\,f_1(t)f_2(t),
\end{eqnarray*}
where $C\in\mathbb{R}$ and where
\begin{eqnarray*}
f_1(t)&=&\Pi_{i=1}^{p}\sin(t+\alpha_i) \\
f_2(t)&=&\Pi_{i=1}^{k}\sin(2t+\beta_i).
\end{eqnarray*}
Where $\alpha_i, \beta_i\in\mathbb{R}$ and 
\begin{eqnarray*}
k+p=2n-1,
\end{eqnarray*}
with $k\leq p$.
Indeed we have:
\begin{lemma} \label{specnc}
Let $\nu$ be a needle candidate defined on $\mathbb{C}P^n$. Let the support of $\nu$ be a geodesic $\gamma$ parametrized by the interval $(0,L)$. Let $U\subset\mathbb{C}P^n$ be a domain such that for every $\gamma(t)$, the volume of the $(2n-1)$-dimensional hyper-surface $U_t\subset U$ which is orthogonal to $\gamma'(t)$ and which is generated by $\{J_1(t),\cdots,J_{2n-1}(t)\}$ is equal to
\begin{eqnarray*}
f(t)= c\sqrt{\det{(<J_i(t),J_k(t)>)_{i,k=1,\cdots, 2n-1}}}.
\end{eqnarray*}
Then, there exists a frame of orthogonal Jacobi fields $\{Q_1(t),\cdots,Q_{2n-1}(t)\}$ along $\gamma$ such that the $(2n-1)$-dimensional volume of $U_t$ is equal to:
\begin{eqnarray*}
C f_1(t)f_2(t),
\end{eqnarray*}
for some $C\in\mathbb{R}$ and $f_1,f_2$ defined as above.
\end{lemma}
\begin{proof}
We consider the exponential map with respect to the submanifold $N=U_0$. For this we need to define the normal bundle $\nu(N)$ which by definition is all the pairs $(x,v)$ such that $x\in N$ and $v$ is orthogonal to $T_x N$. Therefore the exponential map:
\begin{eqnarray*}
exp_N:\nu(N)\to\mathbb{C}P^n,
\end{eqnarray*}
is defined by 
\begin{eqnarray*}
exp_N(x,v)=exp_x(v).
\end{eqnarray*}
For more details on exponential maps with respect to submanifolds we refer to \cite{grey}.

Let $p=\gamma(0)\in N$ where $\gamma(0)'$ is orthogonal to $T_p N$. In $T_p\mathbb{C}P^n$ if we consider the straight line $t\to t\gamma(0)'$ and along this line we consider a linear vector field $X(t)=a+tb$ where $a,b\in T_p N$ are orthogonal vectors which are moved parallel to themselves, then the map $d exp_N$ sends this field to normal Jacobi fields along $\gamma$. It is a classical result and can be verified in textbooks that all orthogonal jacobi fields along $\gamma$ are obtained in this way. Therefore it is sufficient to chose suitable orthogonal linearly independent  $X_i=a_i+b_it$ and take $Y_i=d exp_N X_i$ such that:
\begin{eqnarray*}
\Pi_{i=1}^{2n-1}\Vert Y_i\Vert=C f_1(t)f_2(t).
\end{eqnarray*}
This ends the proof of the Lemma.
\end{proof}

Proposition \ref{procomp} suggests that the needle separation distance (on $\mathbb{C}P^n$) is obtained by the separation distance of the needles of the form $\mu_k=C_k\sin(t)^{2n-k}\cos(t)^k dt$ for some $k\in\{1,\cdots,2n-1\}$. This is already helpful and leaves one to find those (open) subsets of $\mathbb{C}P^n$ which realize the separation distances of the needles of form $\mu_k$ . The next proposition classifies this problem:

\begin{proposition} \label{endcomp}
Let $n\geq 1$. Let $\mu_k=C_k\sin(t)^{2n-k}\cos(t)^k dt$ for 
\begin{eqnarray*}
k\in\{1\cdots,2n-1\},
\end{eqnarray*}
be a probability measure on $[0,\pi/2]$. Let $\kappa_1,\kappa_2>0$ be such that:
\begin{eqnarray*}
\kappa_1+\kappa_2<1.
\end{eqnarray*}
Then, there exist $U_1,U_2 \subset\mathbb{C}P^n$ with
\begin{eqnarray*}
\frac{vol_{2n}(U_i)}{vol_{2n}(\mathbb{C}P^n)}=\kappa_i,
\end{eqnarray*}
for $i=1,2$
such that:
\begin{eqnarray*}
d(U_1,U_2)=Sep(\mu_k,\kappa_1,\kappa_2),
\end{eqnarray*}
if and only if we have $k=2m+1$ for $m\in\{0,\cdots n-1\}$.
\end{proposition}

\begin{proof}
If $k=2m-1$ for $1\leq k\leq n-1$, the subset which realizes the separation distance of $\mu_k$ is either a geodesic ball (and its complementary) or a tube around $\mathbb{C}P^k$ (and its complementary). Indeed, the formula for the volume of a (geodesic) ball in $\mathbb{C}P^n$ can be found in classic Riemannian Geometry textbooks such as in \cite{gallot} and \cite{grey}. Let $x\in\mathbb{C}P^n$, let $r>0$, the ball $B(x,r)\subset \mathbb{C}P^n$ is by definition all the points within distance $\delta$ to the point $x$. The \emph{measure} (normalized volume) is given by:
\begin{eqnarray*}
\frac{vol_{2n}(B(x,r))}{vol_{2n}(\mathbb{C}P^n)}=\frac{\displaystyle\int_{0}^{r}\sin(t)^{2n-1}\cos(t)dt}{\displaystyle\int_{0}^{\pi/2}\sin(t)^{2n-1}\cos(t)dt}.
\end{eqnarray*}

And let $k\geq 1$ and let $\mathbb{C}P^k\subset\mathbb{C}P^n$ be a totally geodesic complex submanifold of $\mathbb{C}P^n$. Let $\delta>0$. Let $\mathbb{C}P^k+\delta$ be a tube of radius $\delta$ around this $\mathbb{C}P^k$. The formula for the volume of such tubes can be found in \cite{grey} and are given by:
\begin{eqnarray*}
\frac{vol_{2n}(\mathbb{C}P^k+r)}{vol_{2n}(\mathbb{C}P^n)}=\frac{\displaystyle\int_{0}^{r}\sin(t)^{2n-2k+1}\cos(t)^{2k-1}dt}{\displaystyle\int_{0}^{\pi/2}\sin(t)^{2n-2k+1}\cos(t)^{2k-1}dt}.
\end{eqnarray*}

Now suppose $k=2m$ and there exist $U_1, U_2\subset\mathbb{C}P^n$ such that for $\kappa_1,\kappa_2>0$ with $\kappa_1+\kappa_2<1$, we have
\begin{eqnarray*}
\frac{vol_{2n}(U_i)}{vol_{2n}(\mathbb{C}P^n)}=\kappa_i,
\end{eqnarray*}
where $i=1,2$. Moreover, we have:
\begin{eqnarray*}
d(U_1,U_2)=Sep(\mu_k,\kappa_1,\kappa_2).
\end{eqnarray*}
Let
\begin{eqnarray*}
\sigma:[0,\pi/2]\to\mathbb{C}P^n,
\end{eqnarray*}
be a geodesic of (maximal) length $\pi/2$ parametrized by its arc length $t$. Let $\{J_1(t),\cdots,J_{2n}(t)\}$ be an orthogonal family of Jacobi fields along $\sigma$ such that $J_{2n}(t)$ coincides with the tangent vector along $\sigma$. We assume for every $t\in[0,\pi/2]$:
\begin{eqnarray*} 
\sqrt{\det(<J_i(t),J_k(t)>)_{i,k=1,\cdots,2n}}&=&\Pi_{i=1}^{2n}\Vert J_i(t)\Vert \\
                                             &=&C_k\sin(t)^{2n-2k}\cos(t)^{2k}.
\end{eqnarray*}


Let $S=\{J_{j1}(t),\cdots,J_{jl}(t)\}$ be the maximal subset of $\{J_1(t),\cdots,J_{2n-1}(t)\}$ such that $J_{j\beta}(0)\neq 0$ for every $\beta\in\{1,\cdots,l\}$. Because of the choice of Jacobi fields, we have $l\geq 1$. Indeed, if for every $i\in\{1,\cdots,2n-1\}$, we have $J_i(0)=0$, then $U_1$ represents a \emph{geodesic} ball in $\mathbb{C}P^n$ and the needle candidate which realizes the separation distance of a geodesic ball (and its complementary) in $\mathbb{C}P^n$ is $\mu=C\sin(t)^{2n-1}\cos(t)dt$ on a \emph{maximal} geodesic parametrized by $(0,\pi/2)$. 

By definition of Jacobi fields, for every $J_{j\beta}\in S$ there exists a (geodesic) variation:
\begin{eqnarray*}
\alpha:[0,\pi/2]\times (-\varepsilon,+\varepsilon)\to \mathbb{C}P^n,
\end{eqnarray*}
such that $\alpha$ is smooth and 
\begin{eqnarray*}
\alpha(t,0)=\gamma(t),
\end{eqnarray*}
and for every $s\in(-\varepsilon,+\varepsilon)$ the curve $\alpha_{s}(t)=\alpha(s,t)$ is a geodesic and such that:
\begin{eqnarray*}
\frac{\partial \alpha}{\partial t}(0,t)=J_{j\beta}(t).
\end{eqnarray*}

Moreover, by (covariant) differentiating the Jacobi fields $J_{j\beta}$ we have 
\begin{eqnarray*}
J_{j\beta}'(0)=0.
\end{eqnarray*}
Hence, the image of the variation for time $s=0$ lies in a totally geodesic submanifold of $\mathbb{C}P^n$. Therefore, the subset $U_1$ is a \emph{tube} around a totally geodesic submanifold of $\mathbb{C}P^n$, \emph{i.e.} we have a $l$-dimensional totally geodesic submanifold of $\mathbb{C}P^n$ such that $U_1=N+\delta$ for some $\delta>0$ and such that: 
\begin{eqnarray*}
\frac{vol_{2n}(N+\delta)}{vol_{2n}(\mathbb{C}P^n)}=\frac{\displaystyle\int_{0}^{\delta}\sin(t)^{2n-2k}\cos(t)^{2k}dt}{\displaystyle\int_{0}^{\pi/2}\sin(t)^{2n-2k}\cos(t)^{2k}dt}.
\end{eqnarray*}

Since every $J_{j\beta}\in S$ is such that the first zero is obtained at distance $\pi/2$, \emph{i.e.} $J_{j\beta}(\pi/2)=0$ and for every $t<\pi/2$, we have $J_{j\beta}(t)\neq 0$, therefore, the totally geodesic submanifold $N$ must be either a point or a $\mathbb{C}P^k\subset\mathbb{C}P^n$. Indeed, for totally geodesic submanifold of $\mathbb{C}P^n$ which consist of \emph{real} projective spaces $\mathbb{R}P^l$, the \emph{first} zero of such Jacobi fields are obtained at distance $\pi/4$. 

This completes the proof.

\end{proof}

\subsection{End proof of Theorem \ref{complexpro}}

Since the maximal focal distance of any submanifold in $\mathbb{C}P^n$ is equal to $\pi/2$, by applying Lemmas \ref{comp}, \ref{specnc} and \ref{dig}, for the purpose of estimating the needle separation distance, we can consider only those needle candidates $\nu$ such that the density $f$ of their pull-back is defined on the interval $(0,L)$ where $L\leq \pi/2$ and $f$ is of the form presented in proposition \ref{procomp}. Hence, let $\mu_1=f(t)dt$ be a probability measure on an interval $(0,L)$ where $L\leq \pi/2$ where:
\begin{eqnarray*}
f(t)=C_1 f_1(t)f_2(t)f_3(t),
\end{eqnarray*}
where
\begin{eqnarray*}
f_1(t)&=&\cos(t+a_1)\cdots\cos(t+a_p)\\ 
f_2(t)&=&\cos(t+b_1)\cdots\cos(t+b_q)\\
f_3(t)&=&\sin(2t+c_1)\cdots \sin(2t+c_s).
\end{eqnarray*}
where $C_1>0$, $s\geq 1$ and
\begin{eqnarray*}
p+q+s=2n-1.
\end{eqnarray*}
For $i\in\{1,\cdots\ \max\{p,q,s\}\}$ we have $a_i, c_i\geq 0$ and $b_i<0$. Moreover we assume $f(0)=f(L)=0$ which is a condition which is verified according to Lemmas \ref{dig} and \ref{comp}. Then, we chose $m_1, m_2, m_3$ such that
\begin{eqnarray*}
m_1+m_2+m_3=2n-1,
\end{eqnarray*}
such that
\begin{eqnarray*}
m_2&\leq& p\\
q&\leq& m_1\\
s&\leq& m_3,
\end{eqnarray*}
and
\begin{eqnarray*}
g(t)&=& C_2\sin(t)^{m_1}\cos(t)^{m_2}\sin(2t)^{m_3} \\
    &=& C_2\sin(t)^{2n-2k+1}\cos(t)^{2k-1},
\end{eqnarray*}
for some $k\in\{1\cdots,n\}$ and some $C_2>0$ such that $\mu_2=g(t)dt$ is a probability measure on $(0,\pi/2)$.

Therefore by applying proposition \ref{procomp} we deduce that for $\kappa_1,\kappa_2>0$ such that
\begin{eqnarray*}
\kappa_1+\kappa_2<1,
\end{eqnarray*}
there exists a $k\in\{1\cdots,n\}$ such that:
\begin{eqnarray*}
N_{\mathbb{C}P^n}(\kappa_1,\kappa_2)=Sep(C_k\sin(t)^{2n-2k+1}\cos(t)^{2k-1}dt,\kappa_1,\kappa_2),
\end{eqnarray*}
for some $k\in\{1,\cdots,n\}$.

Therefore, according to proposition \ref{impoo} we are left to find those subsets of $\mathbb{C}P^n$ which realize the above needle separation distances. This is presented in proposition \ref{endcomp}.

Hence, for every open set $U\subset\mathbb{C}P^n$ we have:
\begin{eqnarray*}
vol_{2n}(U+\delta)\geq vol_{2n}(B+\delta),
\end{eqnarray*}
where $B$ is a (geodesic) ball or a tube around some $\mathbb{C}P^k\subset\mathbb{C}P^n$ where:
\begin{eqnarray*}
vol_{2n}(U)=vol_{2n}(B).
\end{eqnarray*}

This ends the proof of Theorem \ref{complexpro} for the complex projective spaces.

\section{The proof of Theorem \ref{quaterpro} for $\mathbb{H}P^n$}

The proof for the quaternionic projective space is similar to the proof for the complex projective space. Despite this resemblance, we prefer to present the proof in complete details as the result is of great importance.

Quaternionic projective space has a very rich and interesting geometry. It is defined as the quotient of the $(4n+3)$-dimensional sphere by the action of the group $Sp(1)$ which is the group of unit quaternions which is isomorphism to the group structure on the $3$-dimensional sphere $\mathbb{S}^3$. It is a real $4n$-dimensional manifold. We inherit $\mathbb{H}P^n$ with the canonical Riemannian structure (the one which makes the submersion $\mathbb{S}^{4n+3}\to\mathbb{H}P^n$ into a Riemannian submersion) and denote $vol_{4n}$ by the Riemannian volume associated to the canonical Riemannian structure. The sectional curvature of $\mathbb{H}P^n$, similar to $\mathbb{C}P^n$, varies between $1$ and $4$.

Next, we continue with some geometric results.

\begin{proposition} \label{candquat}
Let $n\geq 1$. Let $\nu$ be a needle candidate on $\mathbb{H}P^n$. Let $f$ be the density function of the pull-back of $\nu$ on an interval $A=(0,L)\subset\mathbb{R}$ such that $L\leq \pi/2$. Then 
\begin{eqnarray*}
f(t)=\Pi_{i=1}^{4n-1}f_i(t),
\end{eqnarray*}
where
for every $i\in\{1,\cdots,4n-1\}$ there exists $\alpha_i\in[1,2]$ and $\beta_i\in\mathbb{R}$ such that
\begin{eqnarray*}
f_i(t)=\sin(\alpha_i t+\beta_i).
\end{eqnarray*}
\end{proposition}

\begin{proof}

Let $\nu$ be a needle candidate supported on a geodesic $\gamma\subset \mathbb{H}P^n$ which is parametrized by its arclength $t\in A\subseteq [0,\pi/2]$. Let $f$ be the density function of $\nu$.

We know that the function $f$ is a $\sin^{4n-1}$-concave function.

Similar to the previous section, we are interested in finding more properties for $f$ by considering the special geometry of $\mathbb{H}P^n$.

\begin{lemma} \label{dig2}
Let the function $f$ be as above. Then
\begin{eqnarray*}
f=f_1 f_2
\end{eqnarray*}
where $f_1$ is a function such that $q(t)=f_1^{\frac{1}{3}}$ satisfies the inequality:
\begin{eqnarray*}
q''+4q\leq 0,
\end{eqnarray*}
and $f_2$ is a $\sin^{4n-4}$-concave function.
\end{lemma}

\begin{proof}
The proof of this Lemma is similar to the proof of Lemma \ref{dig}.

The quaternionic projective space has real dimension $4n$. Based on definition \ref{nc} of the needle candidates we know there exists a family of Jacobi fields $J_1(t),\cdots J_{4n-1}(t)$ along $\gamma$ for $t\in A$. As in the definition \ref{nc}, we let $J_{4n}=\gamma'$. We assume this family consists of linearly independent (tangent) vectors. Otherwise, there is nothing to prove. 

There exists a constant $c>0$ and the density of the pull-back of the needle $\nu$, denoted by $f$, equals: 
\begin{eqnarray*}
f(t)= c\sqrt{\det{(<J_i(t),J_k(t)>)_{i,k=1,\cdots, 4n}}}.
\end{eqnarray*}

The Riemannian volume $vol_{4n}$ of $\mathbb{H}P^n$ can be transported by the (inverse of the) exponential map on the tangent space at any point. We know that
\begin{eqnarray*}
\sqrt{\det{(<J_i(t),J_k(t)>)_{i,k=1,\cdots, 4n}}},
\end{eqnarray*}
is the determinant of the \emph{Gram} matrix associated to vectors $J_i(t)\in T_{\gamma(t)}\mathbb{H}P^n$ for $i=1,\cdots,4n$. This quantity is equal to the  (Riemannian) volume of the parallelepiped generated by the linearly independent tangent vectors $\{J_i(t)\}_{i=1}^{4n}$ in the tangent space $T_{\gamma(t)}\mathbb{H}P^n$. We denote by:
\begin{eqnarray*}
[J_1,\cdots,J_{4n}],
\end{eqnarray*}
the parallelepiped, generated by the vectors $J_i$ ($i=1,\cdots,4n$). For every $t\in A$, we denote (abusively) by $vol_{4n}$ the Riemannian volume in the tangent space $T_{\gamma(t)}\mathbb{H}P^n$. Here, the space of Jacobi fields along a geodesic $\gamma$ is a linear vector space of dimension $8n$.

With the same argument as in the proof of Lemma \ref{dig} we consider the \emph{frame} of Jacobi fields $J_i$ to be orthogonal to $\gamma'$ for every $t\in A$ and the family $\{J_i(t)\}$ to form an orthogonal basis of the tangent space at each point on $\gamma(t)$.

Let $t\in A$ and let $\gamma(t)$ be a point on the geodesic $\gamma$. Let $T_{\gamma(t)}\mathbb{H}P^n$ be the tangent space at this point. We have:
\begin{eqnarray*}
vol_{4n}([J_1(t),\cdots,J_{4n}(t)])&=&\Vert J_{4n}(t)\Vert vol_{4n-1}([J_1(t),\cdots,J_{4n-1}]) \\
                                   &=& vol_{4n-1}([J_1(t),\cdots,J_{4n-1}]).
\end{eqnarray*}
Indeed, the geodesic $\gamma$ is parametrized by its arclength, hence 
\begin{eqnarray*}
\Vert J_{4n}\Vert &=&\Vert \gamma'\Vert \\
                  &=& 1.
\end{eqnarray*}
Moreover, 
\begin{eqnarray*}
[J_1(t),\cdots,J_{4n-1}],
\end{eqnarray*}
is the parallelepiped generated by $J_i(t)$ for $i=1,\cdots,4n-1$, and is contained in a $(4n-1)$-dimensional vector subspace of $T_{\gamma(t)}\mathbb{H}P^n$. This parallelepiped is in fact a $(4n-1)$-dimensional rectangle as $\{J_i(t)\}_{i=1}^{4n}$ form an orthogonal basis of the tangent space $T_{\gamma(t)}\mathbb{H}P^n$.

At the point $\gamma(t)$, consider the orthonormal basis:
\begin{eqnarray*}
e_1,e_2 \cdots e_{4n},
\end{eqnarray*}
where there exists an $m\in\{1,\cdots,4n-1\}$ such that
\begin{eqnarray*}
e_m &=&ie_{m-1} \\
e_{m+1}&=&je_{m-1}\\
e_{m+2}&=&ke_{m-1},
\end{eqnarray*}
where $\{i,j,k\}$ are the unit quaternions and $\{e_{m-1},e_{m},e_{m+1},e_{m+2}\}$ form an orthonormal basis of the (one-dimensional) quaternionic vector space $\subset T_{\gamma(t)}\mathbb{H}P^n$. 

Denote this vector sub-space by $T_4$.


We have:
\begin{eqnarray*}
vol_{4n-1}([J_1(t),\cdots J_{2n-1}(t)]=vol_3([J_m,J_{m+1},J_{m+2}]) vol_{4n-4}(P_{4n-4}),
\end{eqnarray*}
where
\begin{eqnarray*}
P_{4n-4}=[J_2(t),\cdots, J_m(t)^{\circ},J_{m+1}(t)^{\circ},J_{m+2}(t)^{\circ},\cdots, J_{4n-1}(t)],
\end{eqnarray*}
is the $(4n-4)$-dimensional rectangle on which $J_m,J_{m+1},J_{m+2}$ are orthogonal. 

By the property of $\mathbb{H}P^n$ and with similar arguments presented in the proof of Lemma \ref{dig}, we conclude that $\{J_m,J_{m+1},J_{m+2}\}$ are Jacobi fields on a $\mathbb{S}^4(4)$ which is a $4$-dimensional sphere of constant curvature $4$. Thus, let 
 \begin{eqnarray*}
f_1(t)=vol_3([J_p(t),J_q(t),J_s(t)]),
\end{eqnarray*}
and
\begin{eqnarray*}
q(t)=f_1(t)^{1/3}.
\end{eqnarray*}
Then, $q(t)$ satisfies the following differential inequality:
\begin{equation} \label{eqn:closerr}
q''+4q\leq 0.
\end{equation}

Let
\begin{eqnarray*}
f_2(t)=vol_{4n-4}(P_{4n-4}).
\end{eqnarray*}

Since the norm of the Jacobi fields $J_i(t)$ for every $i\in\{2,\cdots,4n-1\}$ except $J_m, J_{m+1}, J_{m+2}$ is a $\sin$-concave function for $t\in A$, therefore, according to lemma \ref{comp}, the function $f_2(t)$ is a $\sin^{4n-4}$-concave function. This ends the proof of the lemma.

\end{proof}

Let 
\begin{eqnarray*}
p:\mathbb{S}^{4n+3}\to \mathbb{H}P^n
\end{eqnarray*}
be the Riemannian submersion from the canonical $4n+3$-dimensional sphere to the $n$-dimensional quaternionic projective space.

Let $\gamma\subset \mathbb{H}P^n$ be a (minimal) geodesic on the quaternionic projective space. Let $J$ be a Jacobi field along $\gamma$. There exists a geodesic $\hat{\gamma}$ on $\mathbb{S}^{4n+3}$ and a Jacobi field along this geodesic which projects on $\gamma$ and $J$. According to Proposition (\ref{kr}) we know the general form of orthogonal Jacobi fields along a geodesic on $\mathbb{S}^{4n+3}$. Therefore, by projection on $\mathbb{H}P^n$ and applying lemma \ref{dig2} the proof of the proposition is completed.

\end{proof}

With the same arguments presented for $\mathbb{C}P^n$, here we can also only consider those needle candidates $\nu$ such that the density of their pull-back is given by:
\begin{eqnarray*}
C\,f_1(t)f_2(t),
\end{eqnarray*}
where $C\in\mathbb{R}$ and where
\begin{eqnarray*}
f_1(t)&=&\Pi_{i=1}^{p}\sin(t+\alpha_i) \\
f_2(t)&=&\Pi_{i=1}^{k}\sin(2t+\beta_i).
\end{eqnarray*}
Where $\alpha_i, \beta_i\in\mathbb{R}$ and 
\begin{eqnarray*}
k+p=4n-1,
\end{eqnarray*}
with $3\leq k\leq p$.

We now like to classify the needle candidates on $\mathbb{H}P^n$ which have density of their pull-backs given by $\sin(t)^m\cos(t)^s$ for some $s,m$ with $s+m=4n+2$. This is provided by the next:

\begin{proposition} \label{endquat}
Let $n\geq 1$. Let $\mu_k=C_k\sin(t)^{4n-k+2}\cos(t)^k dt$ for 
\begin{eqnarray*}
k\in\{1\cdots,4n-1\},
\end{eqnarray*}
be a probability measure on $[0,\pi/2]$. Let $\kappa_1,\kappa_2>0$ be such that:
\begin{eqnarray*}
\kappa_1+\kappa_2<1.
\end{eqnarray*}
Then, there exist $U_1,U_2 \subset\mathbb{H}P^n$ with
\begin{eqnarray*}
\frac{vol_{4n}(U_i)}{vol_{4n}(\mathbb{H}P^n)}=\kappa_i,
\end{eqnarray*}
for $i=1,2$
such that:
\begin{eqnarray*}
d(U_1,U_2)=Sep(\mu_k,\kappa_1,\kappa_2),
\end{eqnarray*}
if and only if we have $k=4m+3$ for $m\in\{0,\cdots n-1\}$.
\end{proposition}

\begin{proof}
If $k=4m+3$ for $1\leq m\leq n-1$, the subset which realizes the separation distance of $\mu_k$ is either a geodesic ball (and its complementary) or a tube around $\mathbb{H}P^k$ (and its complementary). Indeed, let $x\in\mathbb{H}P^n$, let $r>0$, the ball $B(x,r)\subset \mathbb{H}P^n$ is by definition all the points within distance $r$ to the point $x$. For $k\geq 1$ let $\mathbb{H}P^k$ be a totally geodesic quaternionic submanifold of dimension $k$ in $\mathbb{H}P^n$ and for $\delta>0$, the tube $\mathbb{H}P^k+\delta$ around this submanifold is the set of points within distance $\delta$ to $\mathbb{H}P^n$.

One may consult \cite{grey} (for the volume of balls in $\mathbb{H}P^n$) and \cite{totgeo} (for the volume of tubes around totally geodesic $\mathbb{H}P^k\subset\mathbb{H}P^n$) and verify that:
\begin{eqnarray*}
\frac{vol_{4n}(B(x,r))}{vol_{4n}(\mathbb{H}P^n)}=\frac{\displaystyle\int_{0}^{r}\sin(t)^{4n-1}\cos(t)^{3}dt}{\displaystyle\int_{0}^{\pi/2}\sin(t)^{4n-1}\cos(t)^{3}dt}.
\end{eqnarray*}
And

\begin{eqnarray*}
\frac{vol_{4n}(\mathbb{H}P^k+\delta)}{vol_{4n}(\mathbb{H}P^n)}=\frac{\displaystyle\int_{0}^{r}\sin(t)^{4n-4k-1}\cos(t)^{4k+3}dt}{\displaystyle\int_{0}^{\pi/2}\sin(t)^{4n-4k-1}\cos(t)^{4k+3}dt}.
\end{eqnarray*}

Now suppose $k\neq 4m+3$ and there exist $U_1, U_2\subset\mathbb{H}P^n$ such that for $\kappa_1,\kappa_2>0$ with $\kappa_1+\kappa_2<1$, we have
\begin{eqnarray*}
\frac{vol_{4n}(U_i)}{vol_{4n}(\mathbb{H}P^n)}=\kappa_i,
\end{eqnarray*}
where $i=1,2$. Moreover, we have:
\begin{eqnarray*}
d(U_1,U_2)=Sep(\mu_k,\kappa_1,\kappa_2).
\end{eqnarray*}
Let
\begin{eqnarray*}
\sigma:[0,\pi/2]\to\mathbb{H}P^n,
\end{eqnarray*}
be a geodesic of (maximal) length $\pi/2$ parametrized by its arc length $t$. Let $\{J_1(t),\cdots,J_{4n}(t)\}$ be an orthogonal family of Jacobi fields along $\sigma$ such that $J_{4n}(t)$ coincides with the tangent vector along $\sigma$. We assume for every $t\in[0,\pi/2]$ we have:
\begin{eqnarray*} 
\sqrt{\det(<J_i(t),J_k(t)>)_{i,k=1,\cdots,4n}}&=&\Pi_{i=1}^{4n}\Vert J_i(t)\Vert \\
                                             &=&C_m\sin(t)^{4n-k+2}\cos(t)^{k}.
\end{eqnarray*}


Let $S=\{J_{j1}(t),\cdots,J_{jl}(t)\}$ be the maximal subset of $\{J_1(t),\cdots,J_{4n-1}(t)\}$ such that $J_{j\beta}(0)\neq 0$ for every $\beta\in\{1,\cdots,l\}$. Because of the choice of Jacobi fields, we have $l\geq 1$. Indeed if for every $i\in\{1,\cdots,4n-1\}$, we have $J_i(0)=0$, then $U_1$ represents a \emph{geodesic} ball in $\mathbb{H}P^n$ and the needle candidate which realizes the separation distance of a geodesic ball in $\mathbb{H}P^n$ (and its complementary) is $\mu=C\sin(t)^{4n-1}\cos(t)^3dt$ on a \emph{maximal} geodesic parametrized by $(0,\pi/2)$. 

By definition of Jacobi fields, for every $J_{j\beta}\in S$ there exists a (geodesic) variation:
\begin{eqnarray*}
\alpha:[0,\pi/2]\times (-\varepsilon,+\varepsilon)\to \mathbb{H}P^n,
\end{eqnarray*}
such that $\alpha$ is smooth and 
\begin{eqnarray*}
\alpha(t,0)=\gamma(t),
\end{eqnarray*}
and for every $s\in(-\varepsilon,+\varepsilon)$ the curve $\alpha_{s}(t)=\alpha(s,t)$ is a geodesic and such that:
\begin{eqnarray*}
\frac{\partial \alpha}{\partial t}(0,t)=J_{j\beta}(t).
\end{eqnarray*}

Moreover, by (covariant) differentiating the Jacobi fields $J_{j\beta}$ we have 
\begin{eqnarray*}
J_{j\beta}'(0)=0.
\end{eqnarray*}
Hence, the image of the variation for time $s=0$ lies in a totally geodesic submanifold of $\mathbb{H}P^n$. Therefore, the subset $U_1$ is a \emph{tube} around a totally geodesic submanifold of $\mathbb{H}P^n$, \emph{i.e.} we have a $l$-dimensional totally geodesic submanifold of $\mathbb{H}P^n$ such that $U_1=N+\delta$ for some $\delta>0$ and such that: 
\begin{eqnarray*}
\frac{vol_{4n}(N+\delta)}{vol_{4n}(\mathbb{H}P^n)}=\frac{\displaystyle\int_{0}^{\delta}\sin(t)^{m}\cos(t)^{s}dt}{\displaystyle\int_{0}^{\pi/2}\sin(t)^{m}\cos(t)^{s}dt}.
\end{eqnarray*}

Since every $J_{j\beta}\in S$ is such that the first zero is obtained at distance $\pi/2$, \emph{i.e.} $J_{j\beta}(\pi/2)=0$ and for every $t<\pi/2$, we have $J_{j\beta}(t)\neq 0$, therefore, the totally geodesic submanifold $N$ must be either a point or a $\mathbb{H}P^k\subset\mathbb{C}P^n$. Indeed, for totally geodesic submanifolds of $\mathbb{H}P^n$ which consist of \emph{real} projective spaces $\mathbb{R}P^l$ or complex projective space $\mathbb{C}P^l$, the \emph{first} zero of such Jacobi fields are obtained at distance $<\pi/2$.  

The proof hence is completed.

\end{proof}

\subsection{End Proof of Theorem \ref{quaterpro}}

Since the maximal focal distance of any submanifold in $\mathbb{H}P^n$ is equal to $\pi/2$, by applying Lemmas \ref{comp} and \ref{dig2}, for the purpose of estimating the needle separation distance, we can consider only those needle candidates $\nu$ such that the density $f$ of their pull-back is defined on the interval $(0,L)$ where $L\leq \pi/2$ and $f$ is of the form presented in proposition \ref{procomp}. Hence, let $\mu_1=f(t)dt$ be a probability measure on an interval $(0,L)$ where $L\leq \pi/2$ where:
\begin{eqnarray*}
f(t)=C_1 f_1(t)f_2(t)f_3(t),
\end{eqnarray*}
where
\begin{eqnarray*}
f_1(t)&=&\cos(t+a_1)\cdots\cos(t+a_p)\\ 
f_2(t)&=&\cos(t+b_1)\cdots\cos(t+b_q)\\
f_3(t)&=&\sin(2t+c_1)\cdots \sin(2t+c_s).
\end{eqnarray*}
where $C_1>0$, $s\geq 3$ and
\begin{eqnarray*}
p+q+s=4n-1.
\end{eqnarray*}
For $i\in\{1,\cdots\ \max\{p,q,s\}\}$ we have $a_i, c_i\geq 0$ and $b_i<0$. Moreover we assume $f(0)=f(L)=0$ which is a condition which is verified according to Lemma \ref{dig2}. Then, we chose $m_1, m_2, m_3$ such that
\begin{eqnarray*}
m_1+m_2+m_3=4n-1,
\end{eqnarray*}
such that
\begin{eqnarray*}
m_2&\leq& p\\
q&\leq& m_1\\
s&\leq& m_3,
\end{eqnarray*}
and
\begin{eqnarray*}
g(t)&=& C_2\sin(t)^{m_1}\cos(t)^{m_2}\sin(2t)^{m_3} \\
    &=& C_2\sin(t)^{4n-4k-1}\cos(t)^{4k+3},
\end{eqnarray*}
for some $k\in\{0\cdots,n-1\}$ and some $C_2>0$ such that $\mu_2=g(t)dt$ is a probability measure on $(0,\pi/2)$.

By applying proposition \ref{procomp} we deduce that for $\kappa_1,\kappa_2>0$ such that
\begin{eqnarray*}
\kappa_1+\kappa_2<1,
\end{eqnarray*}
there exists a $k\in\{1\cdots,n-1\}$ such that:
\begin{eqnarray*}
N_{\mathbb{H}P^n}(\kappa_1,\kappa_2)=Sep(C_k\sin(t)^{4n-4k-1}\cos(t)^{4k+3} dt,\kappa_1,\kappa_2).
\end{eqnarray*}

Therefore, according to Proposition \ref{impoo}, it remains to find those subsets of $\mathbb{H}P^n$ which realize the above (optimal) needle separation distance. This is presented in Proposition \ref{endquat}.

Therefore, we apply the result of Proposition \ref{impoo} and deduce that for every open set $U\subset \mathbb{H}P^n$, we have: 
\begin{eqnarray*}
vol_{4n}(U+\delta)\geq vol_{4n}(B+\delta),
\end{eqnarray*}
where $B$ is either a (geodesic) ball or a tube around a totally geodesic $\mathbb{H}P^k\subset\mathbb{H}P^n$ such that
\begin{eqnarray*}
vol_{4n}(U)=vol_{4n}(B).
\end{eqnarray*}
This ends the proof of Theorem \ref{quaterpro} for the quaternionic projective spaces.


\section{Remarks and Questions}

Here, we list a few remarks as well as some open problems related to the topic of this paper.
\begin{itemize}
\item The first question concerns the sharp isoperimetric inequality on the Cayley projective plane $CaP^2$. We were not able to properly use the geometry of this space in order to study the needles defined upon $CaP^2$. The main reason for this is the fact that there is not a Riemannian submersion from some round sphere to $Cap^2$, and this makes the study of needle candidates more delicate compared to the cases for $\mathbb{R}P^n$, $\mathbb{C}P^n$ and $\mathbb{H}P^n$. It is certainly a very interesting problem of studying needles and the isoperimetry of $CaP^2$ with the methods of the present paper.

\item Keeping the spirit of \emph{positive curvature} and $CD(n-1,n)$-needles, it is interesting to study the sharp isoperimetric inequality on Lens spaces. One is invited to study needle candidates on such spaces further and to examine the optimal needle separation distance for such needles. Consult \cite{va} for the study of isoperimetric inequalities on Lens spaces.

\item Does Klartag's needle decomposition also hold for non-Riemannian metric spaces? This is in comparison with the recent work of Cavalleti and Mondino in \cite{mond}. 

\item Is it possible to generalize Klartag's partition theorem \ref{klartag} to partitions where the measures (\emph{generalized needles}) are supported on higher (co)-dimensional sets? A generalization of needles in higher dimension is required and for the case of the round sphere this is provided in \cite{growst} and further explained in \cite{memwst}.

\item For us, the most interesting problem is to study isoperimetric inequalities on (compact) quotients of the Hyperbolic space $\mathbb{H}^n$. The needles on the negatively- curved Riemannian manifolds should (probably) behave as $\sinh^n$-concave probability measures. As far as we know, the needles are not yet studied on negatively curved Riemannian manifolds. The needle decomposition tool may be an interesting tool in order to study the famous \emph{Cartan-Hadamard} conjecture/theorem (as proof of this long standing conjecture is presented in \cite{ghomi}). This problem consists of isoperimetric inequality on Cartan-Hadamard manifolds.
\end{itemize}

\end{document}